\documentclass[12pt]{amsart}

\usepackage{amssymb, amsfonts}
\usepackage{url}
\usepackage[all,arc]{xy}
\usepackage{amscd}
\usepackage{mathrsfs}
\usepackage{geometry}
\usepackage[colorlinks,citecolor=blue]{hyperref}
\usepackage{comment}
\usepackage{enumerate}
\usepackage{graphicx}
\usepackage{bm}
\usepackage{color}
\usepackage{anyfontsize}
\usepackage{caption}
\usepackage{dirtytalk}

\numberwithin{equation}{section}

\theoremstyle{plain}
\newtheorem{theorem}{Theorem}[section]
\newtheorem{corollary}[theorem]{Corollary}
\newtheorem{lemma}[theorem]{Lemma}
\newtheorem{proposition}[theorem]{Proposition}

\newtheorem{definition}[theorem]{Definition}

\theoremstyle{remark}
\newtheorem{remark}{Remark}[section]

\begin{document}

%\date{\today} 

\title[On free boundary minimal submanifolds in geodesic balls]{On free boundary minimal submanifolds in geodesic balls in $\mathbb H^n$ and $\mathbb S^n_+$}

\author{Vladimir Medvedev}

\address{Faculty of Mathematics, National Research University Higher School of Economics, 6 Usacheva Street, Moscow, 119048, Russian Federation}

\email{vomedvedev@hse.ru}

%\thanks{}

%\subjclass{}

\begin{abstract}
We consider free boundary minimal submanifolds in geodesic balls in the hyperbolic space $\mathbb H^n$ and in the round upper hemisphere $\mathbb S^n_+$. Recently, Lima and Menezes have found a connection between free boundary minimal surfaces in geodesic balls in $\mathbb S^n_+$ and maximal metrics for a functional, defined on the set of Riemannian metrics on a given compact surface with boundary. This connection is similar to the connection between free boundary minimal submanifolds in Euclidean balls and the critical metrics of the functional "the $k$-th normalized Steklov eigenvalue", introduced by Faser and Schoen. We define two natural functionals on the set of Riemannian metrics on a compact surface with boundary. One of these functionals is the high order generalization of the functional, introduced by Lima and Menezes. We prove that the critical metrics for these functionals arise as metrics induced by free boundary minimal immersions in geodesic balls in $\mathbb H^n$ and in $\mathbb S^n_+$, respectively.  We also prove a converse statement. Besides that, we discuss the (Morse) index of free boundary minimal submanifolds in geodesic balls in $\mathbb H^n$ or $\mathbb S^n_+$. We show that the index of the critical spherical catenoid in a geodesic ball in $\mathbb S^3_+$ is 4 and the index of the critical spherical catenoid in a geodesic ball in $\mathbb H^3$ is at least 4. We prove that the index of a geodesic $k$-ball in a geodesic $n$-ball in $\mathbb H^n$ or $\mathbb S^n_+$ is $n-k$. For the proof of these statements we introduce the notion of spectral index similarly to the case of free boundary minimal submanifolds in a unit ball in the Euclidean space.
\end{abstract}

\maketitle

%%%%%%%%%%%%%%%%%%%%%%%%%%%%%%%%%%%%%%%%%%%%%%%%%%%%%%%%%%%%%%%%%%%%%%%%%
% Macros
%%%%%%%%%%%%%%%%%%%%%%%%%%%%%%%%%%%%%%%%%%%%%%%%%%%%%%%%%%%%%%%%%%%%%%%%%

\newcommand\cont{\operatorname{cont}}
\newcommand\diff{\operatorname{diff}}

\newcommand{\dvol}{\text{dA}}
\newcommand{\Ric}{\operatorname{Ric}}
\newcommand{\GL}{\operatorname{GL}}
\newcommand{\myO}{\operatorname{O}}
\newcommand{\myP}{\operatorname{P}}
\newcommand{\eye}{\operatorname{Id}}
\newcommand{\myF}{\operatorname{F}}
\newcommand{\Vol}{\operatorname{Vol}}
\newcommand{\odd}{\operatorname{odd}}
\newcommand{\even}{\operatorname{even}}
\newcommand{\ol}{\overline}
\newcommand{\mye}{\operatorname{E}}
\newcommand{\myo}{\operatorname{o}}
\newcommand{\myt}{\operatorname{t}}
\newcommand{\irr}{\operatorname{Irr}}
\newcommand{\mydiv}{\operatorname{div}}
\newcommand{\re}{\operatorname{Re}}
\newcommand{\im}{\operatorname{Im}}
\newcommand{\can}{\operatorname{can}}
\newcommand{\scal}{\operatorname{scal}}
\newcommand{\tr}{\operatorname{trace}}
\newcommand{\sgn}{\operatorname{sgn}}
\newcommand{\SL}{\operatorname{SL}}
\newcommand{\myspan}{\operatorname{span}}
\newcommand{\mydet}{\operatorname{det}}
\newcommand{\SO}{\operatorname{SO}}
\newcommand{\SU}{\operatorname{SU}}
\newcommand{\specl}{\operatorname{spec_{\mathcal{L}}}}
\newcommand{\fix}{\operatorname{Fix}}
\newcommand{\id}{\operatorname{id}}
\newcommand{\grad}{\operatorname{grad}}
\newcommand{\singsup}{\operatorname{singsupp}}
\newcommand{\wave}{\operatorname{wave}}
\newcommand{\ind}{\operatorname{ind}}
\newcommand{\mynull}{\operatorname{null}}
\newcommand{\inj}{\operatorname{inj}}
\newcommand{\arcsinh}{\operatorname{arcsinh}}
\newcommand{\Spec}{\operatorname{Spec}}
\newcommand{\Ind}{\operatorname{Ind}}
\newcommand{\Nul}{\operatorname{Nul}}
\newcommand{\inrad}{\operatorname{inrad}}
\newcommand{\mult}{\operatorname{mult}}
\newcommand{\Length}{\operatorname{Length}}
\newcommand{\Area}{\operatorname{Area}}
\newcommand{\Ker}{\operatorname{Ker}}
\newcommand{\floor}[1]{\left \lfloor #1  \right \rfloor}

\newcommand\restr[2]{{% we make the whole thing an ordinary symbol
  \left.\kern-\nulldelimiterspace % automatically resize the bar with \right
  #1 % the function
  \vphantom{\big|} % pretend it's a little taller at normal size
  \right|_{#2} % this is the delimiter
  }}

%%%%%%%%%%%%%%%%%%%%%%%%%%%%%%%%%%%%%%%%%%%%%%%%%%%%%%%%%%%%%%%%%%%%%%%%%

 \section{Introduction and main results}
 
 A free boundary minimal submanifold (FBMS for short) $\Sigma$ of a Riemannian manifold with boundary $(M,g)$ is a critical point of the volume functional among all variations leaving the boundary $\partial\Sigma$ on the boundary $\partial M$. Equivalently, $\Sigma$ is an FBMS in $(M,g)$ if its mean curvature vanishes and $\Sigma$ meets the boundary of $M$ orthogonally, i.e., $\Sigma \perp \partial M$ along $\partial\Sigma$. In this paper we consider FBMS in geodesic balls in the upper hemisphere $\mathbb S^n_+$, realized as the upper part of the unit sphere centred at the origin in the Euclidean space $\mathbb E^{n+1}$, and in the hyperbolic space $\mathbb H^n$, realized as the hyperboloid given by $-x_0^2+\sum_{i=1}^nx_i^2=-1$ in $\mathbb R^{n+1}$ with the Minkowskian metric $ds^2=-dx_0^2+\sum_{i=1}^ndx_i^2$. As it was shown in~\cite{lima2023eigenvalue}, the coordinate functions $\Phi_i,~i=0,\ldots,n$ of a $k$-dimensional FBMS $\Sigma$ in the geodesic ball $\mathbb B^n(r)$ in $\mathbb S^n_+$ of radius $0<r<\frac{\pi}{2}$ centred at the point $(1,0,\ldots,0)$ satisfy the following spectral problem
 $$
\begin{cases}
\Delta_{g}\, \Phi_i=k\Phi_i\quad &\textrm{in } \Sigma,\ i=0,1,\ldots,n,\\
\dfrac{\partial \Phi_0}{\partial\eta} =- (\tan r)\Phi_0\quad &\textrm{on } \partial\Sigma,\\
\dfrac{\partial \Phi_i}{\partial\eta} = (\cot r)\Phi_i \quad &\textrm{on } \partial\Sigma,\ i=1,\ldots,n.
\end{cases}
$$
We refer to the functions satisfying this spectral problem as \textit{$k$-Steklov eigenfunctions}. As we can see, the coordinate functions of an FBMS in geodesic balls in $\mathbb S^n_+$ are $k$-Steklov eigenfunctions with eigenvalues either $-\tan r$ or $\cot r$. A similar characterization holds for an FBMS in the geodesic ball $\mathbb B^n(r)$ in $\mathbb H^n$ of radius $r$ centred at the point $(1,0,\ldots,0)$ (see section~\ref{sec:robin}). 

It may seem that the explicit eigenvalue characterization for an FBMS in geodesic balls in $\mathbb H^n$ and $\mathbb S^n_+$ should provide many examples of such submanifolds. However, the list of known examples is quite short. Let us consider them and some of their properties.

\medskip

\textit{Geodesic disks.} Fraser and Schoen proved in~\cite{fraser2015uniqueness} that any free boundary minimal disk $\mathbb B^2(r)$ in a constant curvature ball of any dimension is totally geodesic. It was shown in~\cite{lima2023eigenvalue} that the coordinate functions $\Phi_i,~i=0,\ldots,3$ of the geodesic ball $\mathbb B^2(r)$ with $0<r<\frac{\pi}{2}$ centred at $(1,0,0,0)$ in $\mathbb B^3(r)\subset \mathbb S^3_+$ are of two sorts: $\Phi_0$ is a first 2-Steklov eigenfunction, and $\Phi_j,~j=1,2,3$ are second 2-Steklov eigenfunctions. In fact, as we show in Section~\ref{sec:spec}, a similar result also holds true for free boundary minimal geodesic balls in $\mathbb B^n(r)$ of any dimension not only in $\mathbb S^n_+ $ but also in $\mathbb H^n$. 

\medskip

\textit{Critical spherical catenoids.}  In~\cite{mori1981minimal} Mori built a family of rotational minimal surfaces in $\mathbb H^3$, which he called \textit{spherical catenoids}. Shortly after that this result was generalized by do Carmo and Dajczer in~\cite{do1983rotation}. Particularly, they described a family of rotational minimal hypersurfaces in $\mathbb S^n$ and $\mathbb H^n$. For the case of $\mathbb S^3$, this family is given by $\Phi_{a}:\mathbb R\times \mathbb S^1\to \mathbb S^3$,
 \begin{equation}
\begin{array}{rcl}\label{S-annulus}
\Phi_{a}(s,\theta)&=&\left(\sqrt{\dfrac{1}{2}-a\cos(2s)}\cos\varphi(s), \sqrt{\dfrac{1}{2}-a\cos(2s)}\sin\varphi(s),  \right.\\
&&
 \left.\sqrt{\dfrac{1}{2}+a\cos(2s)}\cos\theta, \sqrt{\dfrac{1}{2}+a\cos(2s)}\sin\theta\right),
\end{array}
\end{equation}
where $-\frac{1}{2}<a\leqslant 0$ is a constant and $\varphi(s)$ is given by
$$
\varphi(s)=\sqrt{\dfrac{1}{4}-a^2}\displaystyle\int_0^s \dfrac{1}{\left(\dfrac{1}{2}-a\cos(2t)\right)\sqrt{\dfrac{1}{2}+a\cos(2t)}}dt.
$$
Similarly, for the case of $\mathbb H^3$, the family is given by $\Phi_{a}:\mathbb R\times \mathbb S^1\to \mathbb H^3$,
 \begin{equation}
\begin{array}{rcl}\label{H-annulus}
\Phi_{a}(s,\theta)&=&\left(\sqrt{a\cosh(2s)+\dfrac{1}{2}}\cosh\varphi(s), \sqrt{a\cosh(2s)+\dfrac{1}{2}}\sinh\varphi(s),  \right.\\
&&
 \left.\sqrt{a\cosh(2s)-\dfrac{1}{2}}\cos\theta, \sqrt{a\cosh(2s)-\dfrac{1}{2}}\sin\theta\right),
\end{array}
\end{equation}
where $a>\frac{1}{2}$ is a constant and $\varphi(s)$ is given by
$$
\varphi(s)=\sqrt{a^2-\frac{1}{4}}\displaystyle\int_0^s \frac{1}{\left(a\cosh(2t)+\dfrac{1}{2}\right)\sqrt{a\cosh(2t)-\dfrac{1}{2}}}dt.
$$
Other parametrizations can be found in~\cite{otsuki1970minimal,krtouvs2014minimal,tuzhilin1991morse}. In~\cite{li2018gap} the authors observed that there exists $-1/2<a\leqslant 0$ (for the spherical case) or $a>1/2$ (for the hyperbolic case) and $s_0\in\mathbb R$ such that $\Phi_a\colon [-s_0,s_0]\times\mathbb S^1\to  \mathbb{B}^3(r)$ is a free boundary minimal annulus in $\mathbb{B}^3(r)$ in $\mathbb S^3_+$ or $\mathbb H^3$, respectively. We call these pieces of spherical catenoids in $\mathbb{B}^3(r)$ the \textit{critical spherical catenoids}. In the same paper the authors also characterize the critical spherical catenoids in geodesic balls in $\mathbb S^n_+$ and $\mathbb H^n$ by a pinching condition. In~\cite{lima2023eigenvalue} Lima and Menezes proved that the critical spherical catenoid in $\mathbb B^3(r) \subset \mathbb S^3_+$ is given by 2-Steklov eigenfunctions in the following way: The 0-component is a first 2-Steklov eigenfunction and the remaining ones are second 2-Steklov eigenfunctions. The proof of this statement can be adapted to the case of the critical spherical catenoid in $\mathbb B^3(r) \subset \mathbb H^3$ (see Section~\ref{sec:apendix}).

The surfaces considered above are analogs of the critical catenoid in a unit ball in the Euclidean 3-space. Fraser and Schoen in~\cite{fraser2016sharp} characterized the metric induced on the critical catenoid from the Euclidean one as the \textit{maximal} metric for the functional \textit{"the first normalized Steklov eigenvalue"} on the set of Riemannian metrics on the annulus. More generally, one can say that this metric is \textit{extremal} for the functional "the first normalized Steklov eigenvalue", i.e., it is a critical point of this functional under one-parameter smooth family of metric deformations (see~\cite{nadirashvili1996berger,el2000riemannian,karpukhin2022laplace}). 

\begin{definition}
\label{extremal:def}
    A metric $g$ on a manifold $\Sigma$ is said to be extremal for some functional $F$ defined on a subset $\mathcal S(\Sigma)$ of the set of Riemannian metrics $\mathcal R(\Sigma)$ if for all one-parameter smooth family of metrics $g(t)\in \mathcal S(\Sigma)$ with
    $g(0) = g$, we have 
    $$ \text{either} \quad F(g(t)) \leqslant F(g) + o(t) \quad \text{or} \quad F(g(t)) \geqslant F(g) + o(t),$$
    as $t \to 0$.
    If $F(g_t)$ has one-sided derivatives by $t$, then, equivalently, one can say that $g$ is extremal for $F$ if
    $$
    \frac{dF(g_t)}{dt}\Big|_{t=0+}\times \frac{dF(g_t)}{dt}\Big|_{t=0-} \leqslant 0,
    $$
    i.e., the one-sided derivatives of $F$ at 0 have opposite signs.
\end{definition}

A particular case of extremal metrics is \textit{maximal metrics}, i.e., such metrics from the set $\mathcal S(\Sigma)$ that the functional $F$ attains its maximum. 

In order to translate the results by Fraser and Schoen about FBMS in a unit ball in the Euclidean space to FBMS in a geodesic ball in $\mathbb S^n_+$ and $\mathbb H^n$, one needs to replace the Stekov problem by the so-called \textit{Robin problem}. For a compact Riemannian manifold $(\Sigma,g)$ with sufficiently smooth boundary $\partial\Sigma$ it can be formulated in the following way
 $$
 \begin{cases}
 \Delta_gu=\alpha u\quad &\text{in } \Sigma,\\
 \dfrac{\partial u}{\partial\eta}=\sigma(g,\alpha) u\quad &\text{on } \partial\Sigma.
 \end{cases}
 $$
If we fix $\alpha\in\mathbb R$, then the real numbers $\sigma(g,\alpha)$ satisfy (see Section~\ref{sec:robin})
  $$\sigma_{0}(g,\alpha) < \sigma_{1}(g,\alpha) \leqslant \ldots \leqslant \sigma_{j}(g,\alpha) \leqslant \ldots \nearrow +\infty.$$
 The multiplicity of each $ \sigma_{j}(g,\alpha)$ is finite. We call them \textit{$\alpha$-Steklov eigenvalues}. If $(\Sigma,g)$ is a compact Riemannian \textit{surface} with sufficiently smooth boundary, then we consider the following two functionals:
$$
\Theta_{r,k}(\Sigma,g) = \left(\sigma_{0}(g,2)\cos^{2} r + \sigma_{k}(g,2)\sin^{2}r \right)|\partial\Sigma|_{g} +2|\Sigma|_{g},\,k\geqslant 1,
$$
and
$$
\Omega_{r,k}(\Sigma,g) = \left(-\sigma_{0}(g,-2)\cosh^{2} r + \sigma_{k}(g,-2)\sinh^{2} r\right)|\partial\Sigma|_{g} +2|\Sigma|_{g},\,k\geqslant 1.
$$
In what follows, we always assume that $0<r<\frac{\pi}{2}$ for $\Theta_{r,k}(\Sigma,g)$ and $r>0$ for $\Omega_{r,k}(\Sigma,g)$ without indicating it explicitly. The functional $\Theta_{r,1}(\Sigma,g)$ was first introduced by Lima and Menezes in~\cite{lima2023eigenvalue}. To the best of our knowledge, the functionals $\Theta_{r,k}(\Sigma,g), \, k\geqslant 2$ and $\Omega_{r,k}(\Sigma,g),\, k\geqslant 1$ have never appeared in the literature before. 

In this paper we study extremal metrics for the above functionals. We need to specify on which sets we will consider them. For the functional $\Omega_{r,k}(\Sigma,g)$ we choose either the whole set $\mathcal R(\Sigma)$ or a conformal class $\mathcal C$ as a subset $\mathcal S(\Sigma)$. The case of the functional $\Theta_{r,k}(\Sigma,g)$ is more subtle. By the reasons, that we explain in Section~\ref{sec:robin}, we consider $\Theta_{r,k}(\Sigma,g)$ either on the subset $\mathcal R_a(\Sigma)$ of $\mathcal R(\Sigma)$, consisting of metrics with the first Dirichlets eigenvalue grater than 2 (the letter "a"stands for "admissible metrics"), or on $\mathcal C\cap \mathcal R_a(\Sigma)$ for some conformal class $\mathcal C$. In the already mentioned paper~\cite{lima2023eigenvalue} the authors showed that $\Theta_{r,1}(\Sigma,g)$ is bounded from above on $\mathcal R_a(\Sigma)$. Moreover, they proved that if a metric $g$ is maximal for $\Theta_{r,1}(\Sigma,g)$ on $\mathcal R_a(\Sigma)$ or in $\mathcal C\cap \mathcal R_a(\Sigma)$, then this metric is necessarily induced by a free boundary minimal (respectively, harmonic) immersion into a geodesic ball $\mathbb B^n(r)$ in $\mathbb S^n_+$. Finally, the authors characterized the metric on the geodesic 2-ball $\mathbb B^2(r)$ in $\mathbb B^n(r)$ in $\mathbb S^n_+$ as a maximal metric for $\Theta_{r,1}(\Sigma,g)$. We extend their results for extremal metrics for the functionals $\Theta_{r,k}(\Sigma,g)$ on $\mathcal R_a(\Sigma)$ and $\mathcal C\cap \mathcal R_a(\Sigma)$ and $\Omega_{r,k}(\Sigma,g)$ on $\mathcal R(\Sigma)$ and $\mathcal C$. Namely, we prove

\begin{theorem}\label{thm:main}
Let $\Sigma$ be a compact surface with boundary.
\begin{itemize}
\item[(I)] Let $g \in \mathcal{R}_a(\Sigma)$ be an extremal metric for $\Theta_{r,k}(\Sigma,g)$ on $\mathcal{R}_a(\Sigma)$ or let $g \in \mathcal{R}(\Sigma)$ be an extremal metric for $\Omega_{r,k}(\Sigma,g)$ on $\mathcal{R}(\Sigma)$. Then there exist independent eigenfunctions $v_0 \in V_{0}(g)$ and $v_1,\ldots,v_n \in V_{k}(g)$, which induce a free boundary minimal isometric immersion $v = (v_0,v_1,\ldots,v_n):(\Sigma,g) \to \mathbb{B}^{n}(r)$ in $\mathbb S^n_+$ or $\mathbb H^n$, respectively.
\item[(II)] Let $g \in \mathcal C\cap\mathcal{R}_a(\Sigma)$ be an extremal metric for $\Theta_{r,k}(\Sigma,g)$ on $\mathcal C\cap\mathcal{R}_a(\Sigma)$ or let $g \in \mathcal C$ be an extremal metric for $\Omega_{r}(\Sigma,g)$ on $\mathcal C$. Then there exist independent eigenfunctions $v_0 \in V_{0}(g)$ and $v_1,\ldots,v_n \in V_{k}(g)$, which induce a free boundary harmonic map $v = (v_0,v_1,\ldots,v_n):(\Sigma,g) \to \mathbb{B}^{n}(r)$ in $\mathbb S^n_+$ or $\mathbb H^n$, respectively.
\end{itemize}
Here $V_{k}(g)$ denotes the eigenspace associated to $\sigma_{k}(g,2)$ in the case of $\Theta_{r,k}(\Sigma,g)$ and the eigenspace associated to $\sigma_{k}(g,-2)$ in the case of $\Omega_{r,k}(\Sigma,g)$.

Conversely, 
\begin{itemize}
\item[(I)] let $g \in \mathcal R_a(\Sigma)$ and assume that there exist $v_0\in V_0(g,2)$ and a collection $(v_1,\ldots, v_n)$ of independent functions in $V_k(g,2)$ such that
\begin{itemize}
\item[(i)] $\displaystyle\sum_{j=0}^n dv_j\otimes dv_j=g$,\\
\item[(ii)]  $\displaystyle\sum_{j=0}^nv_j^2=1$,
\end{itemize} 
or let $g \in \mathcal R(\Sigma)$ and assume that there exist $v_0\in V_0(g,-2)$ and a collection $(v_1,\ldots, v_n)$ of independent functions in $V_k(g,-2)$ such that
\begin{itemize}
\item[(i)] $-dv_0\otimes dv_0+\displaystyle\sum_{j=1}^n dv_j\otimes dv_j=g$,\\
\item[(ii)]  $-v_0^2+\displaystyle\sum_{j=1}^nv_j^2=-1$.
\end{itemize} 
Then $g$ is extremal for $\Theta_{r,k}(\Sigma,g)$ on $\mathcal R_a(\Sigma)$ or for $\Omega_{r,k}(\Sigma,g)$ on $\mathcal R(\Sigma)$, respectively.\\

\item[(II)] Let  $g \in \mathcal C\cap\mathcal{R}_a(\Sigma)$ and assume that there exist $v_0\in V_0(g,2)$ and a collection $(v_1,\ldots, v_n)$ of independent functions in $V_k(g,2)$ such that $\sum_{j=0}^nv_j^2=1$ or let $g \in \mathcal C$ and assume that there exist $v_0\in V_0(g,-2)$ and a collection $(v_1,\ldots, v_n)$ of independent functions in $V_k(g,-2)$ such that $-v_0^2+\sum_{j=1}^nv_j^2=-1$. Then $g$ is extremal for $\Theta_{r,k}(\Sigma,g)$ on $\mathcal C\cap\mathcal{R}_a(\Sigma)$ or for $\Omega_{r}(\Sigma,g)$ on $\mathcal C$, respectively. 
\end{itemize}

\end{theorem}

\begin{remark}
This theorem reveals the motivation for the functionals $\Theta_{r,k}$ and $\Omega_{r,k}$. We explain it for the case of $\Theta_{r,k}$, the explanation for the case of $\Omega_{r,k}$  is absolutely the same. Suppose that a free boundary minimal immersion $\Phi$ of the surface $\Sigma$ in a geodesic ball $\mathbb B^n(r)\subset \mathbb S^n_+$ is given by a $\sigma_0(g,2)$-eigenfunction and $\sigma_k(g,2)$-eigenfunctions. As before, $g$ denotes the induced metric. Then it is not hard to verify that $\Theta_{r,k}(\Sigma,g)=2|\Sigma|_g=2E[\Phi]$, where $E[\Phi]=\frac12\int_\Sigma |d\Phi|^2dA_g$ is the \textit{energy} of $\Phi$. This relation between the functional $\Theta_{r,k}$ and the energy is analogous to the relation between the functionals $\overline\lambda_k$ or $\overline \sigma_k$ and the energy (in the latter case, we mean the energy of the harmonic extension of the map $\Phi$ from its boundary to the whole surface $\Sigma$; see for example~\cite[Table 1]{karpukhin2022laplace}).
\end{remark}

\begin{remark}
As we show in Section~\ref{sec:bounds} below, the functional $\Omega_{r,k}$ is not bounded from above even in the conformal class of any metric on $\Sigma$. As a consequence, $\Omega_{r,k}$ is not bounded from above on $\mathcal R(\Sigma)$, whatever the compact surface $\Sigma$ is. This fact means that it is senseless to discuss maximal metrics for $\Omega_{r,k}$. However, it does make sense to discuss extremal metrics for $\Omega_{r,k}$. We also show that $\Omega_{r,k}$ is bounded from below by 0. Moreover, there exists a sequence of metrics on $\Sigma$, for which $\Omega_{r,k}$ converges to 0, i.e., the lower bound by 0 is sharp. Finally, we notice that, in contrast to $\Omega_{r,k}$, the functional $\Theta_{r,k}$ is not bounded from below on $\mathcal R_a(\Sigma)$ (see Section~\ref{sec:bounds}). However, as we have already mentioned above, Lima and Menezes proved in~\cite{lima2023eigenvalue} that $\Theta_{r,1}$ is bounded from above on $\mathcal R_a(\Sigma)$. As we show in Section~\ref{sec:bounds} (see Proposition~\ref{prop:bounds}), $\Theta_{r,k}, \, k\geqslant 2$ is also bounded from above on $\mathcal R_a(\Sigma)$. More precisely, the following theorem holds

\begin{theorem}
The functional $\Theta_{r,k}(\Sigma,g), \, k\geqslant 1$ is bounded from above on $\mathcal R_a(\Sigma)$, i.e., there exists a constant $C$, depending only on the topology of $\Sigma$ such that for any $g\in \mathcal R_a(\Sigma)$ one has  $\Theta_{r,k}(\Sigma,g) \leqslant Ck\sin^2r$ for $k\geqslant 1$ . The functional $\Omega_{r,k}(\Sigma,g), \, k\geqslant 1$ is nonnegative for any $g \in \mathcal R(\Sigma)$, i.e., it is bounded from below by 0. 
\end{theorem}
\end{remark}

The second part of this paper is dedicated to the (Morse) index of FBMS in geodesic balls in $\mathbb H^n$ and $\mathbb S^n_+$. We denote the index of a $k$-dimensional FBMS $\Sigma$ in $\mathbb B^n(r)$ in $\mathbb H^n$ and $\mathbb S^n_+$ as $\Ind(\Sigma)$. We define \textit{its spectral index} $\Ind_S(\Sigma)$ as \textit{the number of $k$-Steklov eigenvalues less than $\cot r$} if $\Sigma\subset \mathbb S^n_+$ and \textit{the number of $-k$-Steklov eigenvalues less than $\coth r$} if $\Sigma\subset\mathbb H^n$ (for a more detailed definition see Definition~\ref{def:ind_spec} below). Notice that $\Ind_S(\Sigma)\geqslant 1$. Then we show that for free boundary minimal surfaces in geodesic balls in $\mathbb S^n_+$ one has
  $$
  \Ind(\Sigma)\leqslant n\Ind_S(\Sigma)+\dim\mathcal M(\Sigma),
 $$
 where $\mathcal M(\Sigma)$ is the \textit{moduli space of conformal classes} on $\Sigma$. We also show that in the case of free boundary minimal hypersurfaces in geodesic balls in $\mathbb S^n_+$ or $\mathbb H^n$, which are not contained in a hyperplane in $\mathbb R^{n+1}$ passing through the origin, one has
 $$
\Ind(\Sigma) \geqslant \Ind_S(\Sigma)+n.
$$
This implies

\begin{theorem}\label{thm:cat_index}
The critical spherical catenoid in a ball $\mathbb B^3(r)$ of $\mathbb S^3_+$ has index $4$. The critical spherical catenoid in a ball $\mathbb B^3(r)$ of $\mathbb H^3$ has index at least $4$.
\end{theorem}

We conjecture that the index of the critical spherical catenoid in a ball $\mathbb B^3(r)$ of $\mathbb H^3$ is also $4$. Notice that the critical catenoid in a Euclidean 3-ball has also index 4, as it was computed in~\cite{devyver2019index,tran2016index,smith2019morse,medvedev2023index}.

We also find the index of a geodesic $k$-ball. Namely, we prove 

 \begin{theorem}\label{thm:ball_index}
 The index of a $k$-dimensional free boundary minimal geodesic ball in an $n$-dimensional geodesic ball in $\mathbb S^n_+$ or $\mathbb H^n$ equals $n-k$.
 \end{theorem}
  
Finally, as another corollary of the inequalities between the index and the spectral index, we obtain a Devyver-type result (see~\cite[Corollary 7.3]{devyver2019index}).

\begin{corollary}\label{cor:first}
Let $\Sigma \subset \mathbb B^n(r)$ in $\mathbb S^n_+$ or in $\mathbb H^n$ be a free boundary minimal hypersurface of index $n+1$, which is not contained in a hyperplane in $\mathbb R^{n+1}$ passing through the origin. Then $\Ind_S(\Sigma)=1$. In the case, when $\Sigma\subset \mathbb B^3(r)$ in $\mathbb S^3_+$ if additionally $\Sigma$ is a topological annulus, then it is the critical spherical catenoid.
\end{corollary}

\subsection{Open questions} The well-known Fraser-Li conjecture (see~\cite{fraser2014compactness}) states that the critical catenoid is the only embedded free boundary minimal annulus in a three-dimensional unit Euclidean ball. Similarly, we conjecture that the critical spherical catenoids are the only \textit{embedded} free boundary minimal annuli in three-dimensional geodesic balls in $\mathbb S_+^n$ and $\mathbb H^n$. At the same time, we expect that in three-dimensional geodesic balls in $\mathbb S_+^n$ and $\mathbb H^n$ there exist \textit{immersed} free boundary minimal annuli, which are different from the critical spherical catenoids. In the case of FBMS in Euclidean balls, these examples have been recently constructed by Fernandez, Hauswirth, and Mira in~\cite{fernandez2023free}. Also, we conjecture that the critical spherical catenoids are the only immersed FBMS with index 4 in three-dimensional geodesic balls in $\mathbb S_+^n$ and $\mathbb H^n$ (compare with~\cite[Conjecture 1.5.3]{fraser2020extremal} or~\cite[Open Question 6]{li2019free}).

\subsection{Paper organization} The paper is organized in the following way. We start with Section~\ref{sec:robin}, where we collect a necessary background concerning the Robin problem. In Section~\ref{sec:extremal} we prove Theorem~\ref{thm:main}. Section~\ref{sec:bounds} is dedicated to the one-sided unboundedness of the functionals $\Theta_{r,k}(\Sigma,g)$ and $\Omega_{r,k}(\Sigma,g)$ in any conformal class. Here we also show that the functional $\Theta_{r,k}(\Sigma,g)$ is bounded from above on $\mathcal R_a(\Sigma)$, while the functional $\Omega_{r,k}(\Sigma,g)$ is bounded from below by 0 on $\mathcal R(\Sigma)$ and that this lower bound is sharp. The stability questions are studied in Section~\ref{sec:spec}. Here we prove Theorems~\ref{thm:cat_index} and~\ref{thm:ball_index}. In Section~\ref{sec:spec} we also prove index upper and lower bounds in terms of the spectral index. This enables us to deduce Corollary~\ref{cor:first}. Finally, in Section~\ref{sec:apendix} we explain how to adapt the proofs of several statements in the paper~\cite{lima2023eigenvalue} to the setting of the functional $\Omega_{r,k}(\Sigma,g)$.

\subsection{Acknowledgements} The author is grateful to the anonymous reviewers for very useful comments and suggestions. The author would like to express his gratitude to Mikhail Karpukhin for many stimulating discussions and to Iosif Polterovich and Tianyu Ma for remarks on the preliminary versions of the manuscript. Also, the author is grateful to Nicolas Popoff for the discussion about the asymptotic behaviour of the Robin eigenvalues for smooth domains on Riemannian manifolds. Finally, the author would like to thank Asma Hassannezhad for the discussion about upper bounds on $\alpha$-Steklov eigenvalues, which eventually enabled him to prove an upper bound for the functional $\Theta_{r,k}(\Sigma,g)$ on the set of Riemannian metrics. This article is an output of a research project implemented as part of the Basic Research Program at the National Research University Higher School of Economics (HSE University).
 
 \section{Robin problem}\label{sec:robin}
 
 Let $(\Sigma,g)$ be a compact Riemannian manifold with sufficiently smooth boundary $\partial\Sigma$. The \textit{Robin problem} is the following spectral problem
 \begin{align}\label{sys:Robin}
 \begin{cases}
 \Delta_gu=\alpha u\quad &\text{in } \Sigma,\\
 \dfrac{\partial u}{\partial\eta}=\sigma u\quad &\text{on } \partial\Sigma.
 \end{cases}
 \end{align}
 The sign convention for the Laplacian is $\Delta_g=-\mydiv_g\circ\nabla^g$. We call the Robin problem with a fixed $\alpha\in\mathbb R$ the \textit{Steklov problem with frequency $\alpha$} and say that the corresponding $\sigma$ are Steklov eigenvalues with frequency $\alpha$ or simply \textit{$\alpha$-Steklov eigenvalues}. If $\alpha$ is a real number, which does not belong to the spectrum of $\Delta_g$ with the Dirichlet boundary condition, then $\alpha$-Steklov eigenvalues satisfy (see for instance~\cite{girouard2022dirichlet,lima2023eigenvalue}):
 $$
 \sigma_{0}(g,\alpha) < \sigma_{1}(g,\alpha) \leqslant \ldots \leqslant \sigma_{j}(g,\alpha) \leqslant \ldots \nearrow +\infty.
 $$
 Notice also that Steklov eigenvalues with frequency $\alpha$ can be considered as eigenvalues of the \textit{Dirichlet-to-Neumann operator with frequency $\alpha$}, which is defined as $\mathcal D_\alpha u=\frac{\partial \widehat u}{\partial \eta}$, where $\widehat u$ denoted the extension of $u\in \partial\Sigma$ on $\Sigma$ by the solution of the problem $\Delta_g\widehat u-\alpha \widehat u=0$. Following~\cite[Section 2]{lima2023eigenvalue}, we define the domain of $D_\alpha$ as
 \begin{align*}
 \mathrm{Dom}(D_\alpha)=\{&u\in L^2(\partial\Sigma)~|~\exists~\widehat u\in H^1(\Sigma)\colon \mathrm{Tr}\widehat u=u,~\Delta_g\widehat u-\alpha \widehat u=0\\&\text{in the weak sense, and}~D_\alpha u\in L^2(\partial\Sigma)\}.
 \end{align*}
 Here $\mathrm{Tr}\colon H^1(\Sigma)\to  L^2(\partial\Sigma)$ denotes the trace operator.

The $\alpha$-Steklov eigenfunctions have nice properties. 

 \medskip

{\bf Claim 1.} \textit{Eigenfunctions $u_1,u_2$ of the Steklov problem with frequency $\alpha$ with different eigenvalues $\sigma_1, \sigma_2$, respectively, are $L^2(\partial\Sigma)$-orthogonal.}

\begin{proof} This claim is a consequence of the classical spectral theorem but we the proof of it is simple and we consider in here. By Green's formula one has
 $$
 \int_\Sigma\langle\nabla^g u_1, \nabla^g u_2\rangle dA=\int_\Sigma(\Delta_g u_1) u_2dA+\int_{\partial\Sigma} \frac{\partial u_1}{\partial\eta}u_2dL=\alpha\int_\Sigma u_1u_2dA+\sigma_1\int_{\partial\Sigma} u_1u_2dL.
 $$
 On the other hand, similarly, we get
 $$
 \int_\Sigma\langle\nabla^g u_1,\nabla^g u_2\rangle_g dA=\alpha\int_\Sigma u_1u_2dA+\sigma_2\int_{\partial\Sigma} u_1u_2dL.
 $$ 
 Since $\sigma_1\neq\sigma_2$ we obtain that $\int_{\partial\Sigma} u_1u_2dL=0$. 
 \end{proof}

{\bf Claim 2.} (\cite{girouard2022dirichlet,lima2023eigenvalue}) \textit{$\alpha$-Steklov eigenvalues admit the following variational characterization:
\begin{align*}
\displaystyle\sigma_{0}(g,\alpha) = \inf_{u \in \mathrm{Dom}(\mathcal{D}_{\alpha})\setminus\{0\}}\frac{\displaystyle\int_{\Sigma} |\nabla^{g} \widehat{u}|^{2}_{g}\,dA_{g} - \alpha\int_{\Sigma} \widehat{u}^{\,2}\,dA_{g}}{\displaystyle\int_{\partial\Sigma} u^2\,dL_{g}},
\end{align*}
and it is simple. Let $\phi_0$ be a first eigenfunction, then
\begin{equation*}\label{var-sigma1}
\sigma_{1}(g,\alpha) = \inf_{\substack{u \in \mathrm{Dom}(\mathcal{D}_{\alpha})\setminus\{0\} \\ \int_{\partial\Sigma} u \phi_{0}\,dL_{g} = 0}} \frac{\displaystyle\int_{\Sigma} |\nabla^{g} \widehat{u}|^{2}_{g}\,dA_{g} - \alpha\int_{\Sigma} \widehat{u}^{\,2}\,dA_{g}}{\displaystyle\int_{\partial\Sigma} u^2\,dL_{g}}.
\end{equation*} 
Moreover, suppose that $\alpha$ does not belong to the Dirichlet spectrum of $(\Sigma,g)$, then
$$
\displaystyle\sigma_{k}(g,\alpha) =  \inf_{\substack{W\subset\{u\in H^1(\Sigma)\colon\Delta_gu=\alpha u\}\\ \dim W=k+1}}\sup_{0\neq u\in W}\frac{\displaystyle\int_{\Sigma} |\nabla^{g} {u}|^{2}_{g}\,dA_{g} - \alpha\int_{\Sigma} {u}^{\,2}\,dA_{g}}{\displaystyle\int_{\partial\Sigma} u^2\,dL_{g}}.
$$
If $\alpha$ is less than the first Dirichlet eigenvalue of $(\Sigma,g)$, then one can take $W\subset H^1(\Sigma)$.}

\medskip
 
 {\bf Claim 3.} \textit{Suppose that $\alpha$ is less than the first Dirichlet eigenvalue of $(\Sigma,g)$. Then the Courant Nodal Domain Theorem holds, i.e., the number of nodal domains in $\Sigma$ of a $\sigma_k$-eigenfunction is at most $k+1$. Consequently, if an $\alpha$-Steklov eigenfunction does not change its sign on $\Sigma$, then this is a first $\alpha$-Steklov eigenfunction.}
 
 \begin{proof}
 The claim follows from~\cite[Theorem 2.1]{hassannezhad2021nodal} and from the observation after Remark 1 in~\cite{lima2023eigenvalue}. The second part of the claim follows from Claim 1 and the first part of the claim.
 \end{proof}
 
 {\bf Claim 4.} \textit{Let $\mathcal F=\mathcal R(\Sigma)$ if $\alpha \leqslant 0$ and $\mathcal F=\mathcal R_a(\Sigma)$ if $\alpha > 0$. Let $g \in \mathcal{F}$ and consider a smooth path of metrics $t \mapsto g(t)$ such that $g(0) = g$ and $g(t) \in \mathcal{F}$ for all $ t\in [-\varepsilon,\varepsilon]$. Then for any $i\in\mathbb N\cup\{0\}$ the map $t \mapsto \sigma_{i}(g(t),\alpha)$ is Lipschitz in $[-\varepsilon,\varepsilon]$.}
 
 \begin{proof}
The proof for the case, when $\alpha=0$ was given in~\cite{fraser2016sharp}. The proof for the case, when $\alpha>0$ was given in~\cite[Lemma 2, Remark 4]{lima2023eigenvalue}. The proof for the case, when $\alpha<0$ is absolutely similar modulo the uniform boundedness of $\sigma_i(g(t)),~i\in\mathbb N\cup\{0\}$. We explain the arguments for $i=0,1$. Using the fact that the metrics $g(t)$ are uniformly equivalent (see the first part of the proof of~\cite[Lemma 2]{lima2023eigenvalue}), we get that
$$
\int_{\Sigma}|\nabla^{g(t)}u|^2_{g(t)}dA_{g(t)}\leqslant C_1\int_{\Sigma}|\nabla^gu|^2_gdA_{g}, \quad \int_{\Sigma}u^2_{g(t)}dA_{g(t)}\leqslant C_2\int_{\Sigma}u^2_gdA_{g},
$$
and
$$
\int_{\partial\Sigma}u^2dL_{g(t)}\geqslant C_3\int_{\partial\Sigma}u^2_gdL_{g},
$$
where $C_1,C_2$ and $C_3$ are positive constants. Consider a 2-dimensional subspace $W$ of $H^1(\Sigma)$ such that $W\setminus\{0\}\subset\{w\in H^1(\Sigma); w\not\equiv0 \mbox{ on } \partial\Sigma\}$. Then we see that
$$
\sup_{w\in W\setminus\{0\}} \frac{\displaystyle\int_{\Sigma} |\nabla^{g(t)} w|^{2}_{g(t)}\,dA_{g(t)} -\alpha\int_{\Sigma} w^{2}\,dA_{g(t)}}{\displaystyle\int_{\partial\Sigma} w^2\,dL_{g(t)}} \leqslant C_4\sup_{w\in W\setminus\{0\}} \frac{\displaystyle\int_{\Sigma} |\nabla^{g} w|^{2}_{g}\,dA_{g} -\alpha'\int_{\Sigma} w^{2}\,dA_{g}}{\displaystyle\int_{\partial\Sigma} w^2\,dL_{g}}, 
$$
where $\alpha'=\alpha C_2/C_1$ is still negative. Taking the infimum over all $2$-dimensional subspaces in $H ^1(\Sigma)$ and using the variational characterization of $\sigma_1(g(t),\alpha)$ and $\sigma_1(g,\alpha')$ (Claim 2), we get
$$
\sigma_1(g(t), \alpha)\leqslant C_4\sigma_1(g, \alpha')=C.
$$
Moreover, the variational characterization implies $\sigma_1(g(t), \alpha) >\sigma_0(g(t), \alpha) \geqslant 0$ and the uniform boundedness of $\sigma_i\big(g(t), \alpha\big),~i=0,1$ follows.
 \end{proof} 
 
 Below in Section~\ref{sec:bounds} we consider the following type of conformal deformations $g_\varepsilon=e^{2\varphi_\varepsilon}g$, where $g\in \mathcal R(\Sigma)$ and smooth functions $\varphi_\varepsilon$ satisfy 
\begin{gather*}
 supp(\varphi_\varepsilon) \subset \Sigma\setminus(\partial\Sigma)_{\varepsilon/2},\quad \min\{C,1\} \leqslant e^{2\varphi_\varepsilon}\leqslant \max\{C,1\}~\text{on}~\Sigma, \quad~\\\text{and}~ e^{2\varphi_\varepsilon}=C~\text{on}~\Sigma\setminus(\partial\Sigma)_{\varepsilon}.
\end{gather*}
Here $(\partial\Sigma)_{\varepsilon}$ is the tubular $\varepsilon$-neighbourhood (with respect to $g$) of $\partial\Sigma$ and $C$ is a positive constant. We see that the limit metric as $\varepsilon\to 0$ that we denote as $g_0$ is \textit{not} smooth if $C\neq 1$: The limit conformal factor is the constant $C$ everywhere on $\Sigma$ and 1 on $\partial\Sigma$. However, it is not hard to see that for any $\varepsilon>0$ the problem
 $$
 \begin{cases}
 \Delta_{g_\varepsilon}u=\alpha u\quad &\text{in } \Sigma,\\
 \dfrac{\partial u}{\partial\eta_\varepsilon}=\sigma u\quad &\text{on } \partial\Sigma
 \end{cases}
 $$
 is equivalent to 
 \begin{align}\label{sys:var}
 \begin{cases}
 \Delta_{g}u=\alpha e^{2\varphi_\varepsilon} u\quad &\text{in } \Sigma,\\
 \dfrac{\partial u}{\partial\eta}=\sigma u\quad &\text{on } \partial\Sigma.
 \end{cases}
 \end{align}
 Then we define the spectrum of the "metric" $g_0$ as the spectrum of the following problem
 \begin{align}\label{sys:var2}
 \begin{cases}
 \Delta_{g}u=C\alpha u\quad &\text{in } \Sigma,\\
 \dfrac{\partial u}{\partial\eta}=\sigma u\quad &\text{on } \partial\Sigma.
 \end{cases}
 \end{align}
 
 \medskip
 
  {\bf Claim 5.} \textit{For the above family of metrics $(g_\varepsilon)_\varepsilon$ one has}
  $$
  \lim_{\varepsilon\to 0}\sigma_i(g_\varepsilon,\alpha)=\sigma_i(g,C\alpha),~\forall i\in\mathbb N\cup\{0\}.
  $$
  
  \begin{proof}
  The proof is analogous to the proof of~\cite[Lemma 6.2]{karpukhin2020conformally}. The proof of the upper semi-continuity of eigenvalues
  $$
  \limsup_{\varepsilon\to 0}\sigma_i(g_\varepsilon,\alpha)\leqslant \sigma_i(g,C\alpha)
  $$ 
  is absolutely similar to the proof of~\cite[Proposition 1.1]{kokarev2014variational}. So we need to prove the lower semi-continuity of eigenvalues
  $$
  \liminf_{\varepsilon \to 0}\sigma_i(g_\varepsilon,\alpha)\geqslant \sigma_i(g,C\alpha), \quad i\geqslant 0.
  $$ 
  One can find a minimizing subsequence $\varepsilon_n$ such that
  $$
   \lim_{n\to\infty}\sigma_j(g_{\varepsilon_n},\alpha)= \liminf_{\varepsilon\to 0}\sigma_j(g_\varepsilon,\alpha) \quad j\leqslant i.
  $$
  Let $u^l_\varepsilon$ be a $\sigma_l(g_{\varepsilon},\alpha)$-eigenfunction. Then $u^l_\varepsilon$ satisfies~\eqref{sys:var}. Pick $u_\varepsilon \in span\{u^0_{\varepsilon},\ldots,u^j_{\varepsilon}\}$. Then it is not hard to see that $u_\varepsilon$ satisfies $\Delta_gu_\varepsilon =\alpha e^{2\varphi_\varepsilon} u_\varepsilon$. It follows from the elliptic regularity (see for instance~\cite[Chapter 3, Lemma 7.1]{ladyzhenskaya1968linear}) that
  $$
  ||u_\varepsilon||^2_{H^1(\Sigma,g)}\leqslant C\left(||\Delta_g u_\varepsilon||^2_{L^2(\Sigma,g)}+||u_\varepsilon||^2_{L^2(\Sigma,g)}\right),
  $$
  where $C$ is a positive constant. Since $\Delta_{g}u_\varepsilon=\alpha e^{2\varphi_\varepsilon} u_\varepsilon$ and $\min\{C,1\} \leqslant e^{2\varphi_\varepsilon}\leqslant \max\{C,1\}$, we have
  $$
   ||u_\varepsilon||_{H^1(\Sigma,g)}\leqslant C||u_\varepsilon||_{L^2(\Sigma,g)}.
  $$
  Then if we assume that $u_\varepsilon$ are normalized so that $||u_\varepsilon||_{L^2(\Sigma,g)}=1$, we obtain that the sequence $(u_\varepsilon)_\varepsilon$ is bounded in $H^1(\Sigma,g)$ and one can extract a weakly convergent subsequence $(u_{\varepsilon_n})_n$ in $H^1(\Sigma,g)$:
  $$
  u_{\varepsilon_n}\rightharpoonup u_0, \quad \text{as} \, n\to\infty.
  $$
  Particularly, if $u^l_{\varepsilon_n}$ is a $\sigma_l(g_{\varepsilon_n},\alpha)$-eigenfunction of unit $L^2(\Sigma,g)$ norm with $l=0,\ldots,j$, then $u^l_{\varepsilon_n}\rightharpoonup u^l_0$, as $n\to\infty$. We need to prove that $u^l_0$ is an eigenfunction of the problem~\eqref{sys:var2} with eigenvalue $\sigma_l=\lim_{n\to\infty}\sigma_l(g_{\varepsilon_n},\alpha)$. By the Rellich-Kondrachov Theorem, this subsequence is strongly convergent in $L^q(\Sigma,g)$ for any $1 \leqslant q <\infty$. Also, by the Trace Theorem, this subsequence is strongly convergent in $L^2(\partial\Sigma,g)$. Finally, due to the uniform boundedness of $(e^{2\varphi_\varepsilon})_\varepsilon$, we get that the sequence $(e^{2\varphi_{\varepsilon_n}}u^l_{\varepsilon_n})_n$ is also strongly convergent in $L^2(\Sigma,g)$, as $n\to\infty$. Clearly, this sequence converges to $Cu^l_0$. Indeed,
  \begin{align*}
  &\left|\int_\Sigma\left(Cu^l_0-e^{2\varphi_{\varepsilon_n}}u^l_{\varepsilon_n}\right)dA_g\right|\leqslant \\ &||u^l_0-u^l_{\varepsilon_n}||_{L^2(\Sigma,g)}||e^{2\varphi_{\varepsilon_n}}||_{L^2(\Sigma,g)}+||u^l_{\varepsilon_n}||_{L^2(\Sigma,g)} ||C-e^{2\varphi_{\varepsilon_n}}||_{L^2(\Sigma,g)} \to 0,\quad \text{as}\, n\to\infty.
  \end{align*}
   
  Then for any $v\in C^\infty(\Sigma)$ one has
  \begin{align*}
  \int_\Sigma\langle\nabla^gu^l_0,\nabla^gv\rangle_gdA_g&-C\alpha\int_\Sigma u^l_0 vdA_g=\\&\lim_{n\to\infty}\left(\int_\Sigma\langle\nabla^gu^l_{\varepsilon_n},\nabla^gv\rangle_gdA_g-\alpha\int_\Sigma e^{2\varphi_{\varepsilon_n}}u^l_{\varepsilon_n} vdA_g \right)= \\&\lim_{n\to\infty}\left(\sigma_l(g_{\varepsilon_n},\alpha)\int_{\partial\Sigma}u^l_{\varepsilon_n} vdL_g\right)=\sigma_l\int_{\partial\Sigma}u^l_0vdA_g.
    \end{align*}
    Hence, $u^l_0$ is a weak eigenfunction of the problem~\eqref{sys:var2} with eigenvalue $\sigma_l$. This immediately proves the lower semi-continuity for $\sigma_0(g_{\varepsilon_n},\alpha)$. In order to conclude that
    $$
    \lim_{n\to\infty}\sigma_l(g_{\varepsilon_n},\alpha)=\sigma_l\geqslant \sigma_l(g,C\alpha),\quad \forall l=0,\ldots,j,
    $$
    in suffices to prove that an $L^2(\Sigma,g_{\varepsilon_n})$-orthogonal family $\{u^l_{\varepsilon_n}\}_{i=0}^j$ of $\sigma_l(g_{\varepsilon_n},\alpha)$-eigen-functions with $l=0,\ldots,j$ converges to an $L^2(\Sigma,g)$-orthogonal family $\{u^l_0\}_{i=0}^j$ of $\sigma_l$-eigenfunctions with the same $l=0,\ldots,j$. Then the lower-semicontinuity of eigenvalues follows from the inductive argument.
    
    Assume that
    $$
    \int_\Sigma u^l_{\varepsilon_n}u^k_{\varepsilon_n}e^{2\varphi_{\varepsilon_n}}dA_g=\delta_{lk}.
    $$
    As we proved above, the sequence $(u^l_{\varepsilon_n})_n$ converges strongly in $L^{2q}(\Sigma,g)$. Then
    \begin{align*}
      &\left|\int_\Sigma\left(Cu^l_0u^k_0-e^{2\varphi_{\varepsilon_n}}u^l_{\varepsilon_n}u^k_{\varepsilon_n}\right)dA_g\right|\leqslant \\ &||u^l_0u^k_0-u^l_{\varepsilon_n}u^k_{\varepsilon_n}||_{L^q(\Sigma,g)}||e^{2\varphi_{\varepsilon_n}}||_{L^p(\Sigma,g)}+||u^l_{\varepsilon_n}u^k_{\varepsilon_n}||_{L^q(\Sigma,g)} ||C-e^{2\varphi_{\varepsilon_n}}||_{L^p(\Sigma,g)} \to 0,
    \end{align*}
    which converges to 0, as $n\to\infty$. Here $p$ and $q$ are H\"older conjugate. Then the family  $\{u^l_0\}_{i=0}^j$ is orthogonal and the lower semi-continuity of eigenvalues is proved.  
  \end{proof}
  
 In Section~\ref{sec:bounds} we will be also interested in the behaviour of $\sigma_0(g,\alpha)$, when $\alpha \to -\infty$. 
 
 \medskip
 
  {\bf Claim 6.} \textit{One has $\sigma_0(g,\alpha)=\sqrt{-\alpha}(1+o(1))$, as $\alpha \to -\infty$. Particularly, for any $k\in \mathbb N$ one has $\sigma_k(g,\alpha)\to+\infty$, as $\alpha\to-\infty$.}
  
  \begin{proof}
We will use the results, proved in the paper~\cite{bruneau2016negative}. In this paper the authors consider domains with corners on a Riemannian manifold. Notice that we can consider any $(\Sigma,g)$ as a \textit{regular} domain in a larger closed Riemannian surface. The term "regular" means, that the metric on $\Sigma$ is a genuine smooth Riemannian metric. Fix $\sigma$ in problem~\eqref{sys:Robin}. Let $\alpha$ be the first eigenvalue of this problem. Allow $\sigma$ to change and consider $\alpha$ as a function of $\sigma$. Then it follows from~\cite[Theorem 1.4]{bruneau2016negative} that
\begin{align}\label{lim}
\lim_{\sigma\to+\infty}\frac{\alpha(\sigma)}{\sigma^2}=-1.
\end{align}
By~\cite[Proposition 2.7]{hassannezhad2021nodal}, for any $\alpha\in (-\infty,\lambda_1^D(\Sigma,g))$, where $\lambda_1^D(\Sigma,g)$ is the first Dirichlet eigenvalue of $(\Sigma,g)$, there exists a unique $\sigma$ such that the first eigenvalue of problem~\eqref{sys:Robin} with this $\sigma$ is exactly $\alpha$ and $\sigma$ is the first eigenvalue of $\mathcal D_\alpha$, i.e., $\sigma_0(g,\alpha)=\sigma$. Moreover, by~\cite[Proposition 2.6]{hassannezhad2021nodal}, $\alpha$ as a function of $\sigma$ is strictly decreasing and $\alpha(\sigma)\to-\infty$, as $\sigma\to+\infty$. This means that the function $\alpha(\sigma)$ is invertible for large $\sigma$ and the inverse function coincides with $\sigma_0(g,\alpha)$. Then passing to the inverse function to $\alpha(\sigma)$ in~\eqref{lim}, we get
$$
\lim_{\alpha\to-\infty}\frac{\alpha}{(\sigma_0(g,\alpha))^2}=-1.
$$
The second part of the claim follows from the fact that $\sigma_k(g,\alpha) > \sigma_0(g,\alpha)$ for any $k\in \mathbb N$.
  \end{proof}

We proceed with geometric applications of the Robin problem. Very recently, Lima and Menezes in~\cite{lima2023eigenvalue} have shown that FBMS in geodesic balls $\mathbb B(r)$ in $\mathbb S^n_+$ satisfy a certain Robin problem. Below we provide their result completed by the case of $\mathbb H^n$. We explain their proof for the case of $\mathbb H^n$.
  
 \begin{proposition}\label{prop:char}
Let $\Phi: \Sigma \to \mathbb B^n(r)$ in $\mathbb S^n_+$ or $\mathbb H^n$ be an immersion of a $k$-dimensional compact manifold with boundary $\Sigma$ such that $\Phi(\partial\Sigma) \subset \partial\mathbb B^n(r)$. Let $g$ be the induced metric on $\Sigma$. Then $\Phi$ is a free boundary minimal immersion if, and only if, the coordinate functions $\Phi_i$ satisfy:
$$
\begin{cases}
\Delta_{g}\, \Phi_i=k\Phi_i\quad &\textrm{in } \Sigma,\ i=0,1,\ldots,n,\\
\dfrac{\partial \Phi_0}{\partial\eta} =- (\tan r)\Phi_0 \quad &\textrm{on } \partial\Sigma,\\
\dfrac{\partial \Phi_i}{\partial\eta} = (\cot r)\Phi_i \quad &\textrm{on } \partial\Sigma,\ i=1,\ldots,n
\end{cases}
$$
for the case when $\mathbb B^n(r)\subset \mathbb S^n_+$, and:
$$ 
\begin{cases}
\Delta_{g}\, \Phi_i=-k\Phi_i\quad &\textrm{in } \Sigma,\ i=0,1,\ldots,n,\\
\dfrac{\partial \Phi_0}{\partial\eta} =(\tanh r)\Phi_0 \quad &\textrm{on } \partial\Sigma,\\
\dfrac{\partial \Phi_i}{\partial\eta} = (\coth r)\Phi_i \quad &\textrm{on } \partial\Sigma,\ i=1,\ldots,n
\end{cases}
$$
 for the case when $\mathbb B^n(r)\subset\mathbb H^n$.
\end{proposition}

In fact, $\Phi_0=\cos r$ on $\partial\Sigma$ for the case of $\mathbb S^n_+$ and $\Phi_0=\cosh r$ on $\partial\Sigma$ for the case of $\mathbb H^n$.

\begin{proof}
For a fixed $v \in \mathbb R^{n+1}$, we introduce the function $\Phi_{v}(x) = \Phi(x)\cdot v$, where $\cdot$ is the Minkowskian scalar product. It is easy to verify that (see for instance \cite[Formula (2.3)]{markvorsen1989characteristic})
$$-\Delta_{g}\,\Phi_{v} = H\cdot v+k\Phi_{v},$$
which implies that $H = 0$ if, and only if, $\Delta_{g}\,\Phi_{v} + k\Phi_{v} = 0$, for any $v \in \mathbb R^{n+1}$.
A computation of the outward pointing unit normal to $\partial\mathbb B^n(r)$ yields: 
$$N_{\partial \mathbb B(r)} = \frac{1}{\sinh r}\big(x\cosh r - \partial_0\big),$$ 
where $\partial_0,\ldots,\partial_n$ is the holonomic basis in the tangent fields on $\mathbb R^{n+1}$ with coordinate functions $x_0,\ldots, x_n$. Since the immersion $\Phi$ is free boundary, we have 
$$N_{\partial \mathbb B(r)}=\frac{\partial \Phi_{v}}{\partial\eta}=\eta \cdot v.$$ Thus,
\begin{equation*}\label{eq:der.normal}
\frac{\partial \Phi_{v}}{\partial\eta} = \frac{1}{\sinh r}\big(\Phi_{v}\cosh r - \partial_0 \cdot v\big).
\end{equation*}
Substituting $v = \partial_i$, $i=0,1,\ldots,n$, we complete the proof.
\end{proof}

\section{Extremal metrics}\label{sec:extremal}
 
In this section we consider the functional
 \begin{align*}
\Omega_{r,k}(\Sigma,g) &= \left(-\sigma_{0}(g)\cosh^{2} r + \sigma_{k}(g)\sinh^{2} r \right)|\partial\Sigma|_{g} +2|\Sigma|_{g},~k\geqslant 1,
\end{align*}
where for simplicity we use the notation $\sigma_i(g,-2)=\sigma_i(g)$.  Everything in this section also holds true for the functional $\Theta_{r,k}(\Sigma,g)$ modulo the replacement $\mathcal R(\Sigma)$ by $\mathcal R_a(\Sigma)$ (see~\cite{lima2023eigenvalue}, where it was proved for the case of maximal metrics for $\Theta_{r,1}(\Sigma,g)$). 

In~\cite{lima2023eigenvalue} the authors found an explicit expression for the derivative (provided that it exists) of the functional $\Theta_{r,1}(\Sigma,g)$ under one-parameter smooth family of deformations of $g$. Similarly to this computation, we find that 
\begin{align*}
\Omega_{r,k}'(0) = Q_h(u_k) &:= \int_{\Sigma} \Big\langle -\tau\big(\cosh r \, u_0\big)  +\tau\big(\sinh r \, u_k\big) + \, g,h\Big\rangle\, dA_{g} \\ 
&\quad\ + \int_{\partial\Sigma} F\big(\cosh r \, u_0,\sinh r \, u_k\big) h(T,T)\,dL_{g} ,
\end{align*}
where
\begin{align*}
h&=\displaystyle\frac{dg}{dt}\Big\vert_{t=0},\\
\tau(v) &= \frac{\vert\partial\Sigma\vert_{g}}{2}\big(|\nabla^{g} v|^{2}_{g} + 2 v^{2}\big)g - \vert\partial\Sigma\vert_{g}\, dv\otimes dv,\\
F\big(v_{0},v_{k}\big) &= -\frac{\sigma_{0}(g)}{2}\big(\cosh^{2} r - |\partial\Sigma|_{g} v_{0}^{2}\big) + \frac{\sigma_{k}(g)}{2}\big(\sinh^{2} r - |\partial\Sigma|_{g} v_{k}^{2}\big),\label{eq:def.F}
\end{align*}
and $T$ is a unit tangent vector field to $\partial\Sigma$ with respect to the metric $g$. 

For a smooth path of metrics $g(t) = f(t)g \in \mathcal{R}(\Sigma)$ in the conformal class $[g]$, we find that
\begin{align*}
\Omega_{r,k}'(0) = Q_w(u_k) &:= -2|\partial\Sigma|_{g}\int_{\Sigma} \left(\cosh^2 r\,u_{0}^{2}-\sinh^2 r\,u_{k}^{2}-1\right)w dA_{g}\\ 
&\quad\  + \int_{\partial\Sigma} F\big(\cosh r\, u_0, \sinh r \,u_k\big)w dL_{g},
\end{align*}
where $w = f'(0)$.  

Consider the Hilbert spaces
$$\mathcal{H}_{g} := L^2\big(S^2(\Sigma),g\big)\times L^2(\partial\Sigma,g),$$
where $S^2(\Sigma)$ is the space of symmetric $(0,2)$-tensor fields on $\Sigma$, and
$$\mathcal L_g:=L^{2}(\Sigma,g)\times L^{2}(\partial\Sigma,g).$$ Recall that $V_{i}(g)$ denotes the eigenspace associated to $\sigma_{i}(g)$. 

\begin{lemma}\label{lem:HB} Let $\Sigma$ be a compact surface with boundary.
\begin{itemize}
\item[(I)]
Suppose that $g \in \mathcal{R}(\Sigma)$ is extremal for the functional $\Omega_{r,k}(\Sigma,g)$. Then, for any $(h,\psi) \in \mathcal{H}_{g}$ there exists $u_k \in V_k(g)$ such that $||u_k||_{L^{2}(\partial\Sigma,g)} = 1$ and 
$$\Big\langle(h,\psi),\Big(-\tau\big(\cosh r u_0\big)  + \tau\big(\sinh r u_k\big) +  g,F\big(\cosh r u_{0},\sinh r u_{k}\big)\Big)\Big\rangle_{\mathcal{H}_{g}} = 0.$$ \\
\item[(II)] Suppose that $g \in [g]$ is extremal for the functional $\Omega_{r,k}(\Sigma,g)$. Then, for any $(w,\psi) \in \mathcal L_g$ there exists $u_k \in V_k(g)$ such that $||u_k||_{L^{2}(\partial\Sigma,g)} = 1$ and 
$$\Big\langle(w,\psi),\Big(-2|\partial\Sigma|_{g}\left(\cosh^2 r\,u_{0}^{2} -\sinh^2 r\,u_{k}^{2}-1\right),F\big(\cosh r\,u_{0},\sinh r\,u_{k}\big)\Big)\Big\rangle_{\mathcal L^{2}} = 0.$$
\end{itemize}
In both cases $u_0$ is the unique positive $\sigma_{0}(g)$-eigenfunction such that $||u_0||_{L^{2}(\partial\Sigma,g)} = 1$.
\end{lemma}

\begin{proof}
(I) The extremality of $g$ implies that the quadratic form 
$$
V_k(g)\ni u\mapsto \Big\langle(h,\psi),\Big(-\tau\big(\cosh r \, u_0\big)  + \tau\big(\sinh r \, u\big) + \,g,F\big(\cosh r \, u_{0},\sinh r \, u \big)\Big)\Big\rangle_{\mathcal{H}_{g}} 
$$
takes on both nonnegative and nonpositive values (see the proof of~\cite[Lemma 3]{lima2023eigenvalue} for details). Thus, it has at least one isotropic direction.

The proof of (II) is similar. 
\end{proof}

\begin{proposition}\label{prop:char.min} 
Let $\Sigma$ be a compact surface with boundary.
\begin{itemize}
\item[(I)] Suppose $g \in \mathcal{R}(\Sigma)$ is extremal for $\Omega_{r,k}(\Sigma,g)$. Then, there exist independent eigenfunctions $v_0 \in V_{0}(g)$ and $v_1,\ldots,v_n \in V_{k}(g)$ which induce a free boundary  minimal isometric immersion $v = (v_0,v_1,\ldots,v_n):(\Sigma,g) \to \mathbb{B}^{n}(r)\subset \mathbb H^n$.
\item[(II)] Suppose $g \in [g]$ is extremal for  $\Omega_{r,k}(\Sigma,g)$. Then, there exist independent eigenfunctions $v_0 \in V_{0}(g)$ and $v_1,\ldots,v_n \in V_{k}(g)$ which induce a free boundary harmonic map $v = (v_0,v_1,\ldots,v_n):(\Sigma,g) \to \mathbb{B}^{n}(r)\subset \mathbb H^n$.
\end{itemize}
\end{proposition}

The proof, which is the same as for the spherical case given in~\cite{lima2023eigenvalue} up to minor modifications for the hyperbolic case and higher eigenvalues, is postponed to~Subsection~\ref{sub:char.min}. In this section we prove a converse statement to the previous proposition.

\begin{proposition}\label{prop:converse} Let $\Sigma$ be a compact surface with boundary.
\begin{itemize}
\item[(I)] Let $g \in \mathcal R(\Sigma)$ and assume that there exist $v_0\in V_0(g)$ and a collection $(v_1,\ldots, v_n)$ of independent functions in $V_k(g)$ such that
\begin{itemize}
\item[(i)] $-dv_0\otimes dv_0+\displaystyle\sum_{j=1}^n dv_j\otimes dv_j=g$,\\
\item[(ii)]  $-v_0^2+\displaystyle\sum_{j=1}^nv_j^2=-1$.
\end{itemize} 
Then $g$ is extremal for $\Omega_{r,k}(\Sigma,g)$.\\

\item[(II)] Let $g \in \mathcal R(\Sigma)$ and assume that there exist $v_0\in V_0(g)$ and a collection $(v_1,\ldots, v_n)$ of independent functions in $V_k(g)$ such that $-v_0^2+\sum_{j=1}^nv_j^2=-1$. Then $g$ is extremal for $\Omega_{r,k}(\Sigma,g)$ in its conformal class $[g]$.
\end{itemize}
\end{proposition}

\begin{proof}
(I) Observe that the equation  $-v_0^2+\sum_{j=1}^nv_j^2=-1$ after taking the normal derivative yields that 
$$
\sigma_k=(\sigma_k-\sigma_0)v_0^2 \text{ on}\, \partial\Sigma.
$$ 
Notice that $\sigma_k-\sigma_0>0$ and $\sigma_k>0$ (the spectrum for $\alpha=-2$ is non-negative by Claim 2). Hence, $v_0=const\neq 0$ along $\partial\Sigma$. Without loss of generality, we can assume that $v_0>0$ on $\partial\Sigma$. Moreover, it is clear that $v_0\geqslant 1$. Then $v_0=\cosh r$ for some $r$ on $\partial\Sigma$. Then the previous equation implies that $\sigma_k=\sigma_0\coth^2 r$.

Fix this $r$ and consider the functional $\Omega_{r,k}$. Include the metric $g$ in a smooth family of metrics $g(t)$, i.e., $g(0)=0, \, t\in (-\varepsilon,\varepsilon)$. Let $h=\frac{dg(t)}{dt}|_{t=0}$. Consider a set of positive real numbers $t_1,\ldots,t_n\in \mathbb R_{+}$ such that $\sum_{j=1}^n t_j=1$ and define the functions 
$$
u_0=\frac{v_0}{|\partial\Sigma|^{1/2}\cosh r} \quad \text{and} \quad u_j=\frac{v_j}{(t_j|\partial\Sigma|)^{1/2}\sinh r},\, j=1,\ldots,n.
$$ 
Then
\begin{align*}
&\sum_{j=1}^nt_jQ_h(u_j)=\int_\Sigma\left\langle-\frac12\left(|\nabla^gv_0|^2_g+2v_0^2\right)g+\sum_{j=1}^n \frac12\left(|\nabla^gv_j|^2_g+2v_j^2\right)g,h\right\rangle dA_g\\&+\int_\Sigma\langle dv_0\otimes dv_0-\sum_{j=1}^n dv_j\otimes dv_j+g,h\rangle dA_g\\
&+\left(-\frac12\sigma_{0}\cosh^{2} r+\frac12\sigma_{k}\sinh^{2}r\right)\int_{\partial\Sigma}h(T,T)dL_g\\&-\frac{1}{2}\int_{\partial\Sigma}\left(-\sigma_0(g)v_0^2+\sum_{j=1}^n\sigma_k(g)v_j^2\right)h(T,T)dL_g\\
&=-\frac{1}{4}\int_{\partial\Sigma}\left(-\frac{\partial}{\partial\eta}v_0^2+\sum_{j=1}^n\frac{\partial}{\partial\eta}v_j^2\right)h(T,T)dL_g=0.
\end{align*}
Here we used that  $-dv_0\otimes dv_0+\sum_{j=1}^n dv_j\otimes dv_j=g$, hence, $-|\nabla^gv_0|^2_g+\sum_{j=1}^n|\nabla^gv_j|^2_g=2$, and the fact that $\frac{\partial}{\partial\eta}v_0=\sigma_0(g)v_0,\, \frac{\partial}{\partial\eta}v_j=\sigma_k(g)v_j,\, j=1,\ldots,n$ on $\partial\Sigma$, while for all points of $\Sigma$ (including the boundary) one has $-v_0^2+\sum_{j=1}^nv_j^2=-1$. We also use that $-\frac12\sigma_{0}\cosh^{2} r+\frac12\sigma_{k}\sinh^{2}r=0$, since, as we have shown above, $\sigma_k=\sigma_0\coth^2 r$. Thus, there exist $u_{\pm}\in V_k(g)$ such that $\pm Q_h(u_\pm)\leqslant 0$. The inequality $Q_h(u_+)\leqslant 0$ implies that
$$
\lim_{t\to 0_+}\frac{\Omega_{r,k}(\Sigma,g(t))-\Omega_{r,k}(\Sigma,g)}{t}\leqslant 0,
$$
while $Q_h(u_-)\geqslant 0$ implies that
$$
\lim_{t\to 0_-}\frac{\Omega_{r,k}(\Sigma,g(t))-\Omega_{r,k}(\Sigma,g)}{t}\geqslant 0.
$$
Whence $\Omega_{r,k}(\Sigma,g(t))\leqslant \Omega_{r,k}(\Sigma,g)+o(t)$. Then the metric $g$ is extremal for $\Omega_{r,k}(\Sigma,g)$ in $\mathcal R(\Sigma)$.

The proof of (II) is similar. Here we consider a smooth deformation $g(t)=f(t)g\in \mathcal R(\Sigma),\, w=f'(0)$, with $f(0)\equiv1$. Then we define the functions $u_0$ and $u_j,\,j=1,\ldots,n$ as above and consider $\sum_{j=1}^nt_jQ_w(u_j)$. The assumption of the proposition implies that this sum is 0. Arguing as above, we conclude that $g$ is extremal for  for $\Omega_{r,k}(\Sigma,g)$ in its conformal class $[g]$.
\end{proof}

\begin{remark}
The proof of the part (I) in the case of the functional $\Theta_{r,k}$ is similar. Namely, the assumption $\sum_{j=0}^nv_j^2=1$ implies that $\sigma_k=(\sigma_k-\sigma_0)v_0^2$ on $\partial\Sigma$. However, in this case $\sigma_k$ can vanish a priory. Then without loss of generality, one can assume that $v_0\geqslant 0$. Moreover, it is clear that $v_0\leqslant 1$. Then $v_0=\cos r$ for some $r$. Hence, $\sigma_k=-\sigma_0\cot^2 r$. The remaining part of the proof is similar.
\end{remark}

Combining Propositions~\ref{prop:char.min} and~\ref{prop:converse}, we obtain Theorem~\ref{thm:main}.

\begin{remark}
In Section~\ref{sec:spec} we show that the geodesic balls and the critical spherical catenoid in $\mathbb B^n(r)\subset \mathbb H^n$ are embedded isometrically by eigenfunctions $v_0\in V_0(g)$ and $v_i\in V_1(g), i=1,\ldots, n$. Similar results were earlier obtained by Lima and Menezes in~\cite{lima2023eigenvalue} for a geodesic 2-ball and the critical spherical catenoid in $\mathbb B^n(r)\subset \mathbb S^n_+$. Hence, the metrics on these submanifolds are extremal for the functionals $\Theta_{r,1}(\Sigma,g)$ or $\Omega_{r,1}(\Sigma,g)$, respectively. In fact, the metric on a geodesic 2-ball in $\mathbb B^n(r) \subset \mathbb S^n_+$ is \textit{maximal} for the functional $\Theta_{r,1}(\Sigma,g)$ on $\mathcal R_a(\Sigma)$, as it was shown in~\cite{lima2023eigenvalue}.
\end{remark}

\section{Upper and lower bounds for the functionals $\Theta_{r,k}(\Sigma,g)$ and $\Omega_{r,k}(\Sigma,g)$}\label{sec:bounds}

In this section we show that the functionals $\Theta_{r,k}(\Sigma,g)$ and $\Omega_{r,k}(\Sigma,g)$ are not bounded from one side on $\mathcal R_a(\Sigma)$ and $\mathcal R(\Sigma)$, respectively, whatever the compact surface $\Sigma$ is. More precisely, $\Theta_{r,k}(\Sigma,g)$ is not bounded from below $\mathcal R_a(\Sigma)$ and $\Omega_{r,k}(\Sigma,g)$ is not bounded from above even in any conformal class in $\mathcal R(\Sigma)$. However, the functional $\Theta_{r,1}(\Sigma,g)$ is bounded from above on $\mathcal R_a(\Sigma)$, as it was previously shown in~\cite{lima2023eigenvalue}. We show that $\Theta_{r,k}(\Sigma,g), \, k\geqslant 2$ is also bounded from above on $\mathcal R_a(\Sigma)$, while  $\Omega_{r,k}(\Sigma,g)$ is bounded from below (by 0) on $\mathcal R(\Sigma)$. In the end of this section we show that the lower bound for the functional $\Omega_{r,k}(\Sigma,g)$ is sharp.

\subsection{One-sided unboundedness of the functionals $\Theta_r(\Sigma,g)$ and $\Omega_r(\Sigma,g)$} We start our discussion with the following proposition:

\begin{proposition}
The functional $\Theta_{r,k}(\Sigma,g)$ is not bounded from below on $\mathcal R_a(\Sigma)$. The functional $\Omega_{r,k}(\Sigma,g)$ is not bounded from above on $\mathcal R(\Sigma)$.  
\end{proposition} 

\begin{proof}
First we show the unboundedness from below of the functional $\Omega_{r,k}(\Sigma,g)$ on $\mathcal R_a(\Sigma)$. Consider a sequence $(g_n)_{n\in\mathbb N}\subset \mathcal R_a(\Sigma)$, converging smoothly to $g$ with the first Dirichlet eigenvalue 2 (se can always find such a sequence, since $\mathcal R_a(\Sigma)$ is open in $\mathcal R(\Sigma)$ with respect to $C^\infty$-topology). By~\cite[Proposition 2]{lima2023eigenvalue}, for that sequence one has $\sigma_0(g_n,2)\to -\infty$ as $n\to \infty$. However, it is easy to see that $\sigma_k(g_n,2)$ is uniformly bounded from above. Indeed, as it was shown in the proof of Lemma 2 in~\cite{lima2023eigenvalue} (see formula (5.2)), for any function $u\in H^1(\Sigma)$ there exist non-zero constants $C_1$ and $C_2$ such that
$$
\int_{\Sigma} |\nabla^{g_n} {u}|^{2}_{g_n}\,dA_{g_n} \leqslant C_1\int_{\Sigma} |\nabla^{g} {u}|^{2}_{g}\,dA_{g}~\text{and}~\int_{\partial\Sigma} u^2\,dL_{g_n}\geqslant C_2\int_{\partial\Sigma} u^2\,dL_{g}.
$$
Moreover, by Claim 2 for any $(k+1)$-dimensional subspace $W$ of $H^1(\Sigma)$, one has
\begin{align*}
\sigma_k(g_n,2) &\leqslant  \sup_{0\neq u\in W}\frac{\displaystyle\int_{\Sigma} |\nabla^{g_n} {u}|^{2}_{g_n}\,dA_{g_n} - 2\int_{\Sigma} {u}^{\,2}\,dA_{g_n}}{\displaystyle\int_{\partial\Sigma} u^2\,dL_{g_n}} \\  &\leqslant\sup_{0\neq u\in W}\frac{\displaystyle\int_{\Sigma} |\nabla^{g_n} {u}|^{2}_{g_n}\,dA_{g_n}}{\displaystyle\int_{\partial\Sigma} u^2\,dL_{g_n}} \leqslant C\sup_{0\neq u\in W}\frac{\displaystyle\int_{\Sigma} |\nabla^{g} {u}|^{2}_{g}\,dA_{g}}{\displaystyle\int_{\partial\Sigma} u^2\,dL_{g}}=C',
\end{align*}
where $C=C_1/C_2$ and $C'=const$. Then, since $|\Sigma|_{g_n} \to |\Sigma|_{g}$ and $|\partial\Sigma|_{g_n} \to |\partial\Sigma|_{g}$, as $n\to\infty$, we get that $\Omega_{r,k}(\Sigma,g_n) \to -\infty$ as $n\to \infty$.
  
Now we prove the unboundedness from above of the functional $\Omega_{r,k}(\Sigma,g)$ on $\mathcal R(\Sigma)$. Consider the following conformal deformation of a metric $g$ on $\Sigma$: $g_{\varepsilon,\delta}=e^{2\varphi_{\varepsilon,\delta}}g$, where $\varphi_\varepsilon$ are smooth and
\begin{align*}
supp(\varphi_{\varepsilon,\delta}) \subset \Sigma\setminus\partial\Sigma_{\delta/2},\quad &\min\left\{1,\frac{1}{2\varepsilon^2}\right\} \leqslant e^{2\varphi_{\varepsilon,\delta}}\leqslant \max\left\{1,\frac{1}{2\varepsilon^2}\right\}~\text{on}~\Sigma,\\ &\text{and}~e^{2\varphi_{\varepsilon,\delta}}=\frac{1}{2\varepsilon^2}~\text{on}~\Sigma\setminus\partial\Sigma_{\delta}.
\end{align*}
Consider the problem
 $$
 \begin{cases}
 \Delta_gu=-\dfrac{1}{\varepsilon^2} u\quad &\text{in } \Sigma,\\
 \dfrac{\partial u}{\partial\eta}=\sigma u\quad &\text{on } \partial\Sigma.
 \end{cases}
 $$ 
Claim 5 implies that for $i=0,k$
$$
\sigma_i(g_{\varepsilon,\delta},-2) \to \sigma_i(g_{\varepsilon},-\frac{1}{\varepsilon^2}),~\text{as } \delta\to 0. 
$$
By Claim 6 we have that $\sigma_i(g_{\varepsilon},-\frac{1}{\varepsilon^2})=\frac{1}{\varepsilon}(1+o(1))$, as $\epsilon\to 0$ and $\sigma_k(g_{\varepsilon},-\frac{1}{\varepsilon^2})>0$ for sufficiently small $\epsilon$. Hence, $\sigma_k(g_{\varepsilon,\delta},-2)>0$ for sufficiently small $\delta$ and $\varepsilon$. Then 
$$
\Omega_{r,k}(\Sigma,g_{\varepsilon,\delta})>-\sigma_0(g_{\varepsilon,\delta},-2)\cosh^2r|\partial\Sigma|_{g_{\varepsilon,\delta}}+2|\Sigma|_{g_{\varepsilon,\delta}}
$$
for sufficiently small $\delta$ and $\varepsilon$. Moreover,
$$
 |\partial\Sigma|_{g_{\varepsilon,\delta}}= |\partial\Sigma|_{g}\quad \text{and}\quad  |\Sigma|_{g_{\varepsilon,\delta}}=\frac{1}{2\varepsilon^2}|\Sigma|_g+o(\varepsilon),~\text{as}~\varepsilon\to 0.
$$
All together implies
\begin{align*}
\Omega_{r,k}(\Sigma,g_{\varepsilon,\delta})>-\sigma_0(g_{\varepsilon,\delta},-2)\cosh^2r|\partial\Sigma|_{g_{\varepsilon,\delta}}+2|\Sigma|_{g_{\varepsilon,\delta}}\xrightarrow[\delta\to 0]{}\\-\frac{1}{\varepsilon}\big(1+o(1)\big)|\partial\Sigma|_{g}+\frac{1}{\varepsilon^2}|\Sigma|_g+o(\varepsilon)\xrightarrow[\varepsilon\to 0]{} +\infty.
\end{align*}
Hence, the functional $\Omega_{r,k}(\Sigma,g)$ is not bounded from above in $[g]$.
\end{proof}

\subsection{Bounds for $\Theta_{r,k}(\Sigma,g)$ and $\Omega_{r,k}(\Sigma,g)$} However, it turns out that the functional $\Theta_{r,k}(\Sigma,g)$ is bounded from above for any $k\in\mathbb N$ on $\mathcal R_a(\Sigma)$, while $\Omega_{r,k}(\Sigma,g)$ is bounded from below for any $k\in\mathbb N$ on $\mathcal R(\Sigma)$. More precisely, the following proposition holds.

\begin{proposition}\label{prop:bounds}
The functional $\Theta_{r,k}(\Sigma,g), \, k\geqslant 1$ is bounded from above on $\mathcal R_a(\Sigma)$. The functional $\Omega_{r,k}(\Sigma,g), \, k\geqslant 1$ is nonnegative for any $g \in \mathcal R(\Sigma)$, i.e., it is bounded from below by 0. 
\end{proposition}

\begin{proof} \textit{Case of $\Theta_{r,k}(\Sigma,g)$.} The case of $k=1$ was proved in~\cite{lima2023eigenvalue}. The proof provided below works for any $k\in\mathbb N$. It is based on the upper bound from~\cite[Section 5]{grigor2004eigenvalues}:
$$
\lambda_k(\mathcal E-\sigma,\nu)\leqslant \frac{Ck-\sigma_{\delta^2}(\Sigma)}{\nu(\Sigma)},\,k \in \mathbb N,
$$
where $C$ is a positive constant, depending only on the topology of $\Sigma$, $\mathcal E(f,h)=\int_\Sigma \langle \nabla^gf,\nabla^g h\rangle_g\,dA_g$ is the energy form, $\sigma=2dA_g$ is the Lebesgue measure on $\Sigma$, so that $\sigma(f,h)=2\int_\Sigma fh\,dA_g$, and $\nu=dL_g$ is the Lebesgue measure on $\partial\Sigma$. With these data one has $\lambda_k(\mathcal E-\sigma,\nu)=\sigma_k(g,2)$ (see definitions (2.15) and (2.39) in~\cite{grigor2004eigenvalues}, where we take $\mathcal F=\mathrm{Dom}(\mathcal D_2)$), $\nu(\Sigma)=|\partial\Sigma|_g$ and $\sigma_{\delta^2}(\Sigma)=2\delta^2|\Sigma|_g$. Moreover, as it follows from the calculations in~\cite[Section 4]{grigor2004eigenvalues} (see pp. 186-187), one can take $\delta=1$ (since the negative part of $\sigma=2dA_g$ is zero). Hence, we get
$$
\sigma_k(g,2) \leqslant \frac{Ck-2|\Sigma|_g}{|\partial\Sigma|_g},\,k \in \mathbb N.
$$
Then
\begin{align*}
\Theta_{r,k}(\Sigma,g)= \left(\sigma_{0}(g,2)\cos^{2} r + \sigma_{k}(g,2)\sin^{2} r\right)|\partial\Sigma|_{g} +2|\Sigma|_{g} \leqslant Ck\sin^2r.
\end{align*}

\textit{Case of $\Omega_{r,k}(\Sigma,g)$.} First, we observe that Claim 2 implies that
$$
\sigma_0(g,-2)\leqslant \frac{2|\Sigma|_g}{|\partial\Sigma|_g}
$$
by simply taking $u=1$ as a test function for the Rayleigh quotient. We can take $u=1$, since $\alpha=-2$ is less than the first Dirichlet eigenvalue of $(\Sigma,g)$. Then
\begin{align*}
\Omega_{r,k}(\Sigma,g)=& \left(-\sigma_{0}(g,-2)\cosh^{2} r + \sigma_{k}(g,-2)\sinh^{2} r\right)|\partial\Sigma|_{g} +2|\Sigma|_{g} >\\ &\left(-\sigma_{0}(g,-2)\cosh^{2} r + \sigma_{0}(g,-2)\sinh^{2} r\right)|\partial\Sigma|_{g} +2|\Sigma|_{g}=\\&-\sigma_{0}(g,-2)|\partial\Sigma|_{g} +2|\Sigma|_{g} \geqslant 0.
\end{align*}
\end{proof}

\begin{remark}
In fact the estimate on $\lambda_k(\mathcal E-\sigma,\nu)$ allows a generalization to the case, when 2 is at an arbitrary location in the Dirichlet spectrum. Then the estimate depends on the number of Dirichlet eigenvalues less than 2.
\end{remark}

\subsection{Sharpness of the lower bound for $\Omega_{r,k}(\Sigma,g)$} In the end of this section we show that the lower bound for the functional $\Omega_{r,k}(\Sigma,g)$ is sharp. Consider as before the following conformal deformation $g_\varepsilon=e^{2\varphi_\varepsilon}g$ with
\begin{align*}
 supp(\varphi_\varepsilon) \subset \Sigma\setminus(\partial\Sigma)_{\varepsilon/2},~\min\{1,\varepsilon\} \leqslant e^{2\varphi_\varepsilon}\leqslant \max\{1,\varepsilon\}~\text{on}~\Sigma,~\text{and}~ e^{2\varphi_\varepsilon}=\varepsilon~\text{on}~\Sigma\setminus(\partial\Sigma)_{\varepsilon}.
\end{align*}
Consider the problem
 $$
 \begin{cases}
 \Delta_gu=0\quad \text{in } \Sigma,\\
 \dfrac{\partial u}{\partial\eta}=\sigma u\quad \text{on } \partial\Sigma.
 \end{cases}
 $$
Then by Claim 5 $\sigma_i(g_\varepsilon,-2)\to \sigma_i^S(g)$ for $i=0,k$, as $\varepsilon\to 0$, where $\sigma_i^S(g)=\sigma_i(g,0)$ is the $i$-th Steklov eigenvalue of $(\Sigma,g)$. Notice that $\sigma_0^S(g)=0$ and $|\Sigma|_{g_\varepsilon}\to 0$, as $\varepsilon \to 0$. Hence, 
$$
\Omega_{r,k}(\Sigma,e^{2\varphi_\varepsilon}g) \xrightarrow[\varepsilon\to 0]{}  \sigma^S_{k}(g)\sinh^{2} r|\partial\Sigma|_{g}.
$$
Let $(g_\delta)_\delta$ be a family of metrics such that $\sigma^S_k(g_\delta) \to 0$, as $\delta \to \infty$ (it always exists, see for instance~\cite[Section 2.2]{girouard2010hersch}). Consider the functions $\varphi_\varepsilon$ as above. We have
$$
\lim_{\varepsilon\to 0}\lim_{\delta\to\infty}\Omega_{r,k}(\Sigma,e^{2\varphi_\varepsilon}g_\delta)=0.
$$
It shows that the lower bound by 0 for the functional $\Omega_{r,k}(\Sigma,g)$ is sharp. 

 \section{Index bounds}\label{sec:spec}
 
 In this section we compute the index of geodesic balls and the critical spherical catenoids. We also give some general upper and lower bounds on the index of an FBMS in geodesic balls in $\mathbb S^n_+$ and $\mathbb H^n$.  

\subsection{Second variation of volume and energy}\label{sec:second} The quadratic form of the second variation of volume of a $k$-dimensional FBMS $\Sigma$ in $\mathbb B^n(r)$ of $\mathbb S^n_+$ or $\mathbb H^n$ is given by (see for example formula (1.4) in \cite{fraser2020extremal})
\begin{equation}\label{areaS}
S(X,X)=\int_\Sigma\Big(|\nabla^\perp X|^2-\mathcal R(X)\cdot X-\mathcal B(X)\Big)dA-\int_{\partial\Sigma}(B_{\partial \mathbb B^n(r)}(X,X)\cdot \eta)dL,
\end{equation}
where $X$ is a normal vector field on $\Sigma$, $\nabla^\perp X=\sum_{i=1}^k\nabla^\perp_{e_i}X$, $e_1,\ldots,e_k$ is a local orthonormal basis in $\Gamma(T\Sigma)$, $\nabla^\perp$ is the connection in the normal bundle, $\mathcal R$ is the curvature operator, $\mathcal B$ is the Simons operator and $B_{\partial \mathbb B^n(r)}$ is the second fundamental form of $\partial \mathbb B^n(r)$ with respect to the \textit{inward} unit normal. Particularly, for free boundary minimal hypersurfaces in $\mathbb B^n(r)$ in $\mathbb S^n_+$ one gets
\begin{gather*}
S(u,u)=\int_\Sigma\Big(|\nabla_{g}u|^2 -(n-1 + |B|^2)u^2\Big)dA - \cot r\int_{\partial\Sigma}u^2dL=\\
=\int_{\Sigma}\Big(\Delta_{g}u - (n-1)u-|B|^2)u\Big)u dA + \int_{\partial\Sigma}\left(\frac{\partial u}{\partial\eta} -  \cot r\,u\right)u dL.
\end{gather*}
In this formula we suppose that the normal bundle on $\Sigma$ is trivial. If $\nu$ is a unit normal vector field on $\Sigma$, then we can write $X=u\nu$ for a function $u$ on $\Sigma$. Abusing the notation, we use $S(u,u)$ in place of $S(X,X)$, since it is determined by $u$.

 Similarly, for $(n-1)$-dimensional FBMS in $\mathbb B^n(r)$ in $\mathbb H^n$ we have
 \begin{gather*}
S(u,u)=\int_\Sigma\Big(|\nabla_{g}u|^2 -(-(n-1) + |B|^2)u^2\Big)dA - \coth r\int_{\partial\Sigma}u^2dL=\\
=\int_{\Sigma}\Big(\Delta_{g}u + (n-1)u-|B|^2)u\Big)u\, dA + \int_{\partial\Sigma}\left(\frac{\partial u}{\partial\eta} -  \coth r\,u\right)u\, dL.
\end{gather*}
In both cases $B$ denotes the second fundamental of $\Sigma$ in $\mathbb B^{n}(r)$.

Let $V$ be a vector field on $\mathbb B^n(r)$ in $\mathbb S^n_+$ or $\mathbb H^n$. We can consider $V$ as a vector field in $\mathbb R^n$. Then the quadratic form of the second variation of energy is given by (see for example~\cite[Section 2]{lima2022bounds}) 
\begin{equation}\label{energyS}
S_E(V,V) = \int_{\Sigma}\Big(|\nabla V|^2 - k|V|^2\Big)\,dA - \cot r\,\int_{\partial\Sigma}|V|^2\,dL,
\end{equation}
for the case of $\mathbb S^n_+$, and by
\begin{equation}\label{energyH}
S_E(V,V) = \int_{\Sigma}\Big(|\nabla V|^2 +k|V|^2\Big)\,dA - \coth r\,\int_{\partial\Sigma}|V|^2\,dL,
\end{equation}
for the case of $\mathbb H^n$. Here $|\nabla V|^2=\sum_{j=0}^n|\nabla^gV^j|^2$ in the case of $\mathbb S^n_+$ and $|\nabla V|^2=-|\nabla^gV^0|^2+\sum_{j=1}^n|\nabla^gV^j|^2$ in the case of $\mathbb H^n$.

\begin{remark}\label{rem:v0}
The vector fields $V$ on $\mathbb R^{n+1}$ in the definition of $S_E$ have to satisfy some additional constraint equation, namely, $V\cdot \Phi=0$, where $\Phi$ is the position vector of $\Sigma$ in $\mathbb R^{n+1}$ and $\cdot$ is either the Euclidean or Minkowskian scalar product. Moreover, the free boundary condition implies that the flow generated by such a field $V$ has to preserve the boundary of $\mathbb B^n(r)$. This means that the $V^0$ component of the field $V$ vanishes on $\partial \mathbb B^n(r)$ (and in particular on $\partial\Sigma$).  
\end{remark}

 \begin{definition}\label{def:ind}
The \emph{(Morse) index} of an FBMS $\Sigma$ in $\mathbb B^n(r)$ of $\mathbb S^n_+$ or $\mathbb H^n$ is the maximal dimension of a subspace in $\Gamma(N\Sigma)$, on which the quadratic form $S$ is negative definite. Similarly, the \emph{energy index} of an FBMS $\Sigma$ in $\mathbb B^n(r)$ of $\mathbb S^n_+$ or $\mathbb H^n$ is the maximal dimension of a subspace in $\Gamma(T\mathbb R^n)$, satisfying the restrictions in Remark~\ref{rem:v0}, on which the quadratic form $S_E$ is negative definite.
 \end{definition}
 
 We use the notation $\Ind(\Sigma)$ for the index of $\Sigma$ and $\Ind_E(\Sigma)$ for the energy index of $\Sigma$.

 \subsection{Spectral index} We start this subsection with the definition of the spectral index for an FBMS in $\mathbb B^n(r)\subset \mathbb S^n_+$ or $ \mathbb H^n$. As a prototype, we take the definition of the spectral index for an FBMS in a unit ball in the Euclidean space, which was introduced by Karpukhin and M\'etras in~\cite{karpukhin2022laplace}.
 
 \begin{definition}\label{def:ind_spec}
  We define \emph{the spectral index} $\Ind_S(\Sigma)$ as the number of $k$-Steklov eigenvalues (counted with multiplicities), less than $\cot(r)$, for a $k$-dimensional FBMS $\Sigma$ in $\mathbb B^n(r)\subset \mathbb S^n_+$ and as the number of $-k$-Steklov eigenvalues (counted with multiplicities), less than $\coth r$, for a $k$-dimensional FBMS $\Sigma$ in $\mathbb B^n(r)\subset \mathbb H^n$. Equivalently, the spectral index of a free boundary minimal immersion $\Phi\colon \Sigma \to \mathbb B^n(r)\subset \mathbb S^n_+$ or $\mathbb H^n$, where $\dim\Sigma=k$, is defined as the maximal dimension of a subspace $V\subset C^\infty(\Sigma)$ such that the quadratic form 
 \begin{align}\label{ind_spec}
S_S(\varphi,\varphi)=\int_\Sigma|\nabla^g \widehat\varphi|_g^2dv_g\pm k\int_\Sigma \widehat\varphi^2dv_g-\int_{\partial \Sigma} |\nabla_\eta \Phi|_g\varphi^2ds_g
  \end{align}
  is negative definite. Here, as before, $\widehat\varphi$ denotes the extension of $\varphi$ in $C^\infty(\partial\Sigma)$ to $C^\infty(\Sigma)$ by the solution of equation $\Delta_g\widehat\varphi\pm k\widehat\varphi=0$. The sign $"-"$ corresponds to the case of $\mathbb S^n_+$ and the sign $"+"$ to the case of $\mathbb H^n$.
 \end{definition}
 
 \begin{remark}
Notice that $ |\nabla_\eta \Phi|_g=\cot r$ in the spherical case and $\coth r$ in the hyperbolic case. 
 \end{remark}
 
 The following proposition almost directly follows from the previous definition.
   
\begin{proposition}\label{prop:spec_ind}
For a $k$-dimensional FBMS $\Sigma$ in $\mathbb B(r)\subset \mathbb S^n_+$ one has
 $$
 \Ind_E(\Sigma)\leqslant n\Ind_S(\Sigma).
 $$
 \end{proposition}
 
 \begin{proof} We apply the same method of the proof as in~\cite{karpukhin2021index}(see also \cite{medvedev2023index} for the Steklov case).  
 
Let $V$ be a maximal negative space of the quadratic form $S_S$, i.e., $\dim V=\Ind_S(\Sigma)$. Assume that 
$$
n\Ind_S(\Sigma) <\Ind_E(\Sigma).
$$
Then there exists a vector field $X$ such that $S_E(X,X)<0$ and whose components $X^i,~i=1,\ldots, n$ are $L^2(\partial\Sigma)$-orthogonal to any function $f\in V$, i.e., $S_S(X^i,X^i)\geqslant 0$ for any $i=1,\ldots, n$. Recall that $\widehat f$ denotes the extension of the function $f$ on $\partial\Sigma$ by the solution of the equation $\Delta_g \widehat u-k\widehat u=0$ on $\Sigma$. Observe that
 \begin{align}\label{ineq:hat}
 \int_{\Sigma}\Big(|\nabla^g \widehat f|^2 - k\widehat f^2\Big)\,dA \leqslant  \int_{\Sigma}\Big(|\nabla^g f|^2 - kf^2\Big)\,dA,
 \end{align}
 by the property of the operator $f\to \widehat f$. So, one has
 $$
  \int_{\Sigma}\Big(|\nabla^g \widehat X^i|^2 - k(\widehat X^i)^2\Big)dA \leqslant  \int_{\Sigma}\Big(|\nabla^g X^i|^2 - k(X^i)^2\Big)dA,\, i=0,1,\ldots,n
 $$
  Moreover, by Remark~\ref{rem:v0}, the coordinate $X^0$ vanishes on $\partial\Sigma$. Hence, $S_S(X^0,X^0)=0$. Then one has
 $$
0>S_E(X,X)\geqslant \sum_{i=0}^nS_S(X^i, X^i) \geqslant 0,
$$ 
where $S_S$ is the quadratic form in the definition of the spectral index. We arrive at a contradiction.
\end{proof}

\begin{remark}
The case of an FBMS $\Sigma$ in $\mathbb B(r)\subset \mathbb H^n$ turns out to be more delicate because of the signature of the Minkowskian metric. For this reason one cannot apply the same arguments as in the proof of the previous proposition directly to this case in order to conclude that $\Ind_E(\Sigma)\leqslant n\Ind_S(\Sigma)$. However, we conjecture, that this inequality holds true for the hyperbolic case as well as for the spherical case.
\end{remark}

In~\cite{lima2022bounds} Lima proved a general upper bound on the index of a compact $2$-dimensional FBMS in \textit{any} ambient Riemannian manifold. His result states that
 $$
   \Ind(\Sigma)\leqslant  \Ind_E(\Sigma)+\dim\mathcal M(\Sigma),
 $$
 where $\mathcal M(\Sigma)$ is the moduli space of conformal structures on $\Sigma$. Combining Lima's inequality with Proposition~\ref{prop:spec_ind}, we obtain
  
 \begin{theorem}\label{thm:ind_M}
  Let $\Sigma$ be a $2$-dimensional FBMS in $\mathbb B(r)\subset \mathbb S^n_+$. Then
 \begin{align*}
  \Ind(\Sigma)&\leqslant n\Ind_S(\Sigma)+\dim\mathcal M(\Sigma).
 \end{align*}
 \end{theorem}
 
 \begin{remark}
 The same inequality also holds true in the flat case (see \cite[Theorem 1.6]{medvedev2023index}).
 \end{remark}
 
 \subsection{Geodesic balls} First, we consider $k$-dimensional free boundary minimal geodesic balls in $n$-dimensional ($n>k$) geodesic balls in $\mathbb S^n_+$ and $\mathbb H^n$. We show 
 
 \begin{theorem}
 The index of a $k$-dimensional free boundary minimal geodesic ball in an $n$-dimensional geodesic ball in $\mathbb S^n_+$ or $\mathbb H^n$ equals $n-k$.
 \end{theorem}
 
 \begin{proof}
Consider the spherical case. The hyperbolic case is absolutely similar. Let $\mathbb B^k(r)$ denote a geodesic $k$-ball. Without loss of generality, assume that $\mathbb B^k(r)$ lies in the plane $\Pi:=\{x\in \mathbb R^{n+1}~|~x_1=\ldots=x_{n-k}=0\}$, i.e., $\mathbb B^k(r)=\mathbb B^n(r)\cap \Pi$. Obviously, the fields $\partial_1,\ldots,\partial_{n-k}$ decrease the area of $\mathbb B^k(r)$. Then the second variation of volume $S$ is negative definite on $span\{\partial_1,\ldots,\partial_{n-k}\}$. Hence, $\Ind(\mathbb B^k(r))\geqslant n-k$. We want to show that $\Ind(\mathbb B^k(r))\leqslant n-k$. Notice that $N_p\mathbb B^k(r)=span\{(\partial_1)_{| p},\ldots,(\partial_{n-k})_{| p}\}$ for any point $p\in \mathbb B^k(r)$. Then any normal to $\mathbb B^k(r)$ field $X$ takes form $X=\sum_{i=1}^{n-k}X^i\partial_i$ for some functions $X^i,~i=1,\ldots,n-k$. Hence, formula~\eqref{areaS} for the quadratic form of the second variation of volume computed on $X$ implies
\begin{align*}
S(X,X)=&\sum_{i=1}^{n-k}\left(\int_{\mathbb B^k(r)}\left(|\nabla^gX^i|^2_g-k(X^i)^2\right)dA-\cot r\int_{\partial\mathbb B^k(r)}(X^i)^2dL\right).
\end{align*}
Indeed, $\mathbb B^k(r)$ is totally geodesic, hence, the Simons operator vanishes, $\mathcal R(X)\cdot X=\sum_{i=1}^k\langle R(e_i,X)X,e_i\rangle=k|X|^2$, where $R$ is the Riemann tensor of $\mathbb S^n$, $e_1,\ldots,e_k$ is a local orthonormal basis in $\Gamma(T\mathbb B^k(r))$, and it is easy to compute that $\nabla^\perp_{e_i}X=e_i(X^j)\partial_j$, since $\nabla^\perp_{e_i}\partial_j=(\nabla^{\mathbb S^n}_{e_i}\partial_j)^\perp=(e_i(\partial_j)-B_{\mathbb S^n}(e_i,\partial_j))^\perp=0$, where $B_{\mathbb S^n}$ is the second fundamental form of $\mathbb S^n$, which is umbilic in $\mathbb E^{n+1}$, $e_i(\partial_j)$ is the Lie derivative of the constant field $\partial_j$ along $e_i$, and $\nabla^{\mathbb S^n}$ is the Levi-Civita connection on $\mathbb S^n$.

Now, we use inequality~\eqref{ineq:hat} to conclude that 
$$
S(X,X) \geqslant \sum_{i=1}^{n-k}S_S(X^i, X^i).
$$
Then the same arguments as in Proposition~\ref{prop:spec_ind} imply that
$$
\Ind(\mathbb B^k(r))\leqslant (n-k)\Ind_S(\mathbb B^k(r))
$$
and the theorem follows from Proposition~\ref{prop:spec_ball} below.
 \end{proof}
 
 \begin{remark}
 In the hyperbolic case we use inequality
 $$
  \int_{\Sigma}\Big(|\nabla^g f|^2 + kf^2\Big)\,dA \leqslant  \int_{\Sigma}\Big(|\nabla^g f|^2 + kf^2\Big)\,dA
 $$
 in place of inequality~\eqref{ineq:hat}. Here $\widehat f$ denotes the extension of the function $f$ on $\partial\Sigma$ by the solution of the equation $\Delta_g \widehat u+k\widehat u=0$ on $\Sigma$. It is important to emphasize that \emph{we do not deal here with the signature of the Minkowskian metric}. This is why the arguments in the proof of the previous theorem work equally for the spherical and hyperbolic cases.
 \end{remark}
  
 \begin{proposition}\label{prop:spec_ball}
One has $\Ind_S(\mathbb B^k(r))=1$ for $\mathbb B^k(r) \subset \mathbb B^n(r) \subset \mathbb S^n_+$ or $\mathbb H^n$. 
\end{proposition}

\begin{proof}

Consider the spherical case. The hyperbolic case is similar (see below). The parametrization of the ball $\Phi: [-r, r]\times \mathbb S^{k-1}\to \mathbb B^k(r)$ is given by 
\begin{align*}
&\Phi_0(t,\theta_1,\ldots,\theta_{k-1}) =\cos t, \\
&\Phi_1(t,\theta_1,\ldots,\theta_{k-1})=\sin t \cos\theta_1,\\
&\Phi_2(t,\theta_1,\ldots,\theta_{k-1})=\sinh t \sin \theta_1\cos\theta_2,\\
&\ldots\\
&\Phi_{k-2}(t,\theta_1,\ldots,\theta_{k-1})=\sin t\sin\theta_1\sin\theta_2\ldots \sin\theta_{k-2}\cos\theta_{k-1},\\
&\Phi_{k-1}(t,\theta_1,\ldots,\theta_{k-1})=\sin t\sin\theta_1\sin\theta_2\ldots \sin\theta_{k-2}\sin\theta_{k-1},
\end{align*}
where $\theta_i\in [0,\pi]$ for $i=1,\ldots k-2$ and $\theta_{k-1}\in [0,2\pi)$. Then the Laplacian of the induced metric takes form
\begin{align*}
-\Delta_gf=\partial_{tt}f+(k-1)(\cot t)\partial_tf-\Delta_{\mathbb S^{k-1}(\sin t)}f,
\end{align*}
where the last term is the Laplacian of the round sphere of radius $\sin t$ that we denote as $\mathbb S^{k-1}(\sin t)$.

By Proposition~\ref{prop:char}, the functions $\Phi_i,~i=0,\ldots,k-1$ are eigenfunctions. Moreover, by Claim 3  $\Phi_0$ is a $\sigma_0( \mathbb B^k(r),k)$-eigenfunction. We need to show that the remaining functions $\Phi_i,~i=1,\ldots,k-1$ are $\sigma_1( \mathbb B^k(r),k)$-eigenfunctions. 

Since the Laplace-eigenfunctions $\{\phi_i\}_{i=0}^\infty$ of $\mathbb S^{k-1}(\sin t)$ --- which are also known as \textit{spherical harmonics} --- form an $L^2\big(\mathbb S^{k-1}(\sin t)\big)$-orthonormal basis, any function $f\in L^2(\mathbb B^k(r))$ decomposes as $f(t,x)=\sum_{i=0}^\infty a_i(t)\phi_i(x)$, where $t\in [-r,r]$ and $x\in \mathbb S^{k-1}(\sin t)$. Suppose that $f$ is a $\sigma_1(\mathbb B^k(r),k)$-eigenfunction. Then 
\begin{align}\label{system}
\begin{cases}
\Delta_gf-kf=0~&\text{in}~\mathbb B^k(r),\\
\dfrac{\partial f}{\partial t}=\sigma_1(\mathbb B^k(r),k)f~&\text{on}~\partial\mathbb B^k(r)
\end{cases}
\end{align} 
yields
\begin{align*}
\begin{cases}
a_i''(t)+(k-1)(\cot t)a_i'(t)+(k-\lambda_i(t))a_i(t)=0,~t\in[-r,r],\\
a_i'(\pm r)=\sigma_1(\mathbb B^k(r),k)a_i(\pm r)
\end{cases}
\end{align*}
for any $i\in\mathbb N\cup\{0\}$. Here $\lambda_i(t)$ is the $i$-th Laplace eigenvalue of $\mathbb S^{k-1}(\sin t)$. Moreover, it is not hard to see that the system~\eqref{system} is satisfied if, and only if, each function $a_i(t)\phi_i(x)$ is a $\sigma_1(\mathbb B^k(r),k)$-eigenfunction. It follows from Claims~1 and~3 that any $\sigma_1(\mathbb B^k(r), k)$-eigenfunction has exactly 2 nodal domains. However, the only spherical harmonics on $\mathbb S^{k-1}(\sin t)$, which have 2 nodal domains, are those, which are from the  $\lambda_1(t)$-eigenspace (see for example~\cite[Section 11.2]{bateman1953higher}). Then the only non-zero terms in the expansion of $f$ are those with $i=0,1,\ldots m$, where $m=k-1$ is the multiplicity of $\lambda_1(t)=\frac{k-1}{\sin^2t}$. Considering the equation 
$$
a_0''(t)+(k-1)(\cot t)a_0'(t)+ka_0(t)=0
$$
yields that $a_0(t)=const$. Indeed, $\cos t$ satisfies this equation. Then we look for a second solution $y$, which is linearly independent of $\cos t$, in the form $y(t)=z(t)\cos t$, where $z$ satisfies
$$
(\cos t)z''(t)+\left(-2\sin t+\frac{(k-1)\cos^2t}{\sin t}\right)z'(t)=0.
$$
The solution to this equation is
$$
z(t)=C_1\frac{\sin^{2-k}t_2F_1(\frac32,1-\frac{k}{2};2-\frac{k}{2};\sin^2t)}{2-k}+C_2,
$$
where $C_1$ and $C_2$ are some constants, and $_2F_1(a,b;c;x)$ is the hypergeometric function. It has a nonessential singularity at $x=0$. The case when $k=2$ was already considered in~\cite[Theorem 3]{lima2023eigenvalue}, so we can assume that $k>2$. Then the solution $y$ does not extend continuously at $t=0$. Hence, $a_0$ is a constant multiple $\cos t$. But it has one single nodal domain. Hence, $a_0\equiv0$.

Further, consider  the equation
$$
a_i''(t)+(k-1)(\cot t)a_i'(t)+\left(k-\frac{k-1}{\sin^2t}\right)a_i(t)=0,~i=1,\ldots m.
$$
Notice that $\sin t$ satisfies this equation. We look for a solution $y$, which is linearly independent of $\sin(t)$ in the form, $y(t)=z(t)\sin t$. Then $z$ satisfies
$$
(\sin t)z''(t)+(k+1)(\cos t)z'(t)=0.
$$
Solving this equation, we get
$$
z(t)=C_1\frac{\sin^{-k}t_2F_1(\frac12,-\frac{k}{2};1-\frac{k}{2};\sin^2t)}{k}+C_2,
$$
for some constants $C_1$ and $C_2$. Then, as we have just discussed above, the solution $y$ does not extend continuously at $t=0$. Thus, $a_i(t)=C_i\sin t,~i=1,\ldots m$ for some constants $C_i$. Therefore, $\Phi_l=(\sin t)\phi_l$, where $\phi_l(\theta_1,\ldots,\theta_{k-1})=\sin\theta_1\sin\theta_2\ldots\cos\theta_l,$ $1\leqslant l \leqslant {k-1}$, is the $l$-th component of the standard basis in the space of $\lambda_1(t)$-eigenfunctions, are $\sigma_1(\mathbb B^k(r),k)$-eigenfunctions. But  $\Phi_l$ has eigenvalue $\cot(r)$. Thus, $\Phi_l$ is a $\sigma_1(\mathbb B^k(r),k)$-eigenfunction and $\sigma_1(\mathbb B^k(r),k)=\cot r$. This concludes the proof for the spherical case.

To get the parametrization of the ball $\Phi: [-r, r]\times \mathbb S^{k-1}\to \mathbb B^k(r)$ for the hyperbolic case, one needs to replace $\cos t$ and $\sin t$ in the above parametrization for the spherical case by $\cosh t$ and $\sinh t$, respectively. Then the Laplacian of the induced metric takes form
$$
-\Delta_gf=\partial_{tt}f+(k-1)(\coth t)\partial_tf-\Delta_{\mathbb S^{k-1}(\sinh t)}f,
$$
where $\mathbb S^{k-1}(\sinh t)$ is the sphere of radius $\sinh t$. Similarly to the spherical case, any function $f\in L^2(\mathbb B^k(r))$ decomposes as $f(t,x)=\sum_{i=0}^\infty a_i(t)\phi_i(x)$, where $t\in [-r,r]$ and $x\in \mathbb S^{k-1}(\sinh t)$. Then $f$ is a $\sigma_1(\mathbb B^k(r),-k)$-eigenfunction if, and only if,
\begin{align*}
\begin{cases}
\Delta_gf+kf=0~&\text{in}~\mathbb B^k(r),\\
\dfrac{\partial f}{\partial t}=\sigma_1(\mathbb B^k(r),-k)f~&\text{on}~\partial\mathbb B^k(r),
\end{cases}
\end{align*} 
which implies
\begin{align*}
\begin{cases}
a_i''(t)+(k-1)(\coth t)a_i'(t)-(k+\lambda_i(t))a_i(t)=0,~t\in[-r,r],\\
a_i'(\pm r)=\sigma_1(\mathbb B^k(r),-k)a_i(\pm r)
\end{cases}
\end{align*}
As in the spherical case, we conclude that the only non-zero terms in the expansion of $f$ are those with $i=0,1,\ldots m$, where $m=k-1$ is the multiplicity of $\lambda_1(t)=\frac{k-1}{\sinh^2t}$. Consider the equation
$$
a_0''(t)+(k-1)(\coth t)a_0'(t)-ka_0(t)=0.
$$
It is not hard to verify that $\cosh t$ is a solution of this equation. For a second linearly independent solution $y(t)=z(t)\cosh t$ one has
$$
(\cosh t)z''(t)+\left(2\sinh t+\frac{(k-1)\cosh^2t}{\sinh t}\right)z'(t)=0.
$$
Solving it, we get
$$
z(t)=C_1\frac{\sinh^{2-k} t_2F_1(\frac32,1-\frac{k}{2};2-\frac{k}{2};-\sinh^2t)}{2-k}+C_2,
$$
if $k>2$. Here $C_1$ and $C_2$ are some constants. If $k=2$, then
$$
z(t)=C_1\left(\frac{1}{\cosh t}+\log\left(\tanh\frac{t}{2}\right)\right)+C_2,
$$
where $C_1$ and $C_2$ are constants. Then for any $k>1$, as in the spherical case, we conclude that $y$ does not extend continuously at $t=0$. Thus, $a_0$ is a constant multiplied by $\cosh t$, which is impossible since it has one single nodal domain. Therefore, $a_0\equiv0$.

Finally, consider  the equation
$$
a_i''(t)+(k-1)(\coth t)a_i'(t)-\left(k+\frac{k-1}{\sinh^2t}\right)a_i(t)=0,~i=1,\ldots m.
$$
Obviously, $\sinh t$ satisfies this equation. For a second linearly independent solution $y(t)=z(t)\sinh t$ we have
$$
(\sinh t)z''(t)+(k+1)(\cosh t)z'(t)=0.
$$
Solving this equation, we get
$$
z(t)=C_1\frac{\sinh^{-k}t_2F_1(\frac12,-\frac{k}{2};1-\frac{k}{2};-\sinh^2t)}{k}+C_2,
$$
where $C_1$ and $C_2$ are some constants. Arguing as above, we conclude that  $y$ does not extend continuously at $t=0$. Hence, $a_i(t)=C_i\sinh t,~i=1,\ldots m$ for some constants $C_i$. Then $\Phi_l=(\sinh t)\phi_l$, where $\phi_l$ is the $l$-component of the standard basis in the space of $\lambda_1(t)$-eigenfunctions, is a $\sigma_1(\mathbb B^k(r),-k)$-eigenfunction for any $l=1,\ldots k-1$. Thus, $\sigma_1(\mathbb B^k(r),-k)=\coth r$. The proof is concluded.

\end{proof}

 \subsection{Non-geodesic submanifolds}
 
 The following theorem is an adaptation of Theorem~1.5 in~\cite{medvedev2023index} to the case of geodesic balls in $\mathbb S^n_+$ or $\mathbb H^n$.
  
 \begin{theorem}\label{thm:ind+n}
Let $\Sigma$ be a free boundary minimal hypersurface in $\mathbb B^n(r)$ in $\mathbb S^n_+$ or $\mathbb H^n$ which is not contained in a
hyperplane in $\mathbb R^{n+1}$ passing through the origin. Then one has
$$
\Ind(\Sigma) \geqslant \Ind_S(\Sigma)+n.
$$
\end{theorem}
 
\begin{proof}
 First we consider the spherical case. Since $\Sigma$ is a free boundary minimal hypersurface in $\mathbb B^n(r)$, then the coordinate functions $u_1,\ldots,u_n$ are $(n-1)$-Steklov eigenfunctions with eigenvalue $\cot r$ and $u_0$ has eigenvalue $-\tan r$. Note that  $u_0,u_1,\ldots,u_n$ are linearly independent, as soon as $\Sigma$ is not contained in a
hyperplane in $\mathbb R^{n+1}$ passing through the origin. Suppose that $\Ind_S(\Sigma)=k+1$, i.e., there are $k$ linearly independent $(n-1)$-Steklov eigenfunctions $\varphi_1,\ldots,\varphi_k$ with eigenvalues $\sigma_i<\cot r,\, i=1,\ldots,k$, respectively, plus the eigenfunction $u_0$. Without loss of generality, one can assume that $\varphi_1,\ldots,\varphi_k$ are orthonormal with respect to the $L^2(\partial\Sigma)$-norm. Notice also that $u_0$ is $L^2(\partial\Sigma)$-orthogonal to $u_j,~\forall j=1,\ldots,n$, since the eigenvalue of $u_0$ is different from the eigenvalue of any of $u_j$. Consider $V=span\{\varphi_1,\ldots,\varphi_k,u_0,u_1,\ldots,u_n\}$. One can see that $\dim V=k+1+n$. We claim that the index form $S$ is negative definite on $V$. Indeed, let $\psi\in V$, i.e., $\psi=\sum_{i=1}^k\alpha_i\varphi_i+\sum_{j=0}^n\beta_ju_j$. Since $\Sigma$ is a hypersurface, the index form $S$ on $\psi$ reads:
\begin{gather}\label{form}
S(\psi,\psi)=\int_\Sigma \Big(\Delta_g\psi-(n-1)\psi-|B|^2\psi\Big)\psi dA+\int_{\partial\Sigma}\left(\frac{\partial \psi}{\partial \eta}-\cot r\,\psi\right)\psi dL.
\end{gather}
Obviously, $\Delta_g\psi-(n-1)\psi=0$, since it is a linear combination of $(n-1)$-Steklov eigenfunctions. Moreover,  
$$
\frac{\partial \psi}{\partial \eta}=\sum_{i=1}^k\alpha_i\sigma_i\varphi_i-\tan r\,\beta_0u_0+\cot r\sum_{j=1}^n\beta_ju_j~\text{on $\partial\Sigma$}.
$$
One may easily check that
\begin{gather}\label{form1}
\int_{\partial\Sigma}\frac{\partial \psi}{\partial \eta}\psi dL=|\partial\Sigma|_g\sum_{i=1}^k\alpha^2_i\sigma_i-\tan r\,\beta_0^2\int_{\partial\Sigma}u_0^2dL+\cot r\int_{\partial\Sigma}\left(\sum_{j=1}^n\beta_ju_j\right)^2dL.
\end{gather}
Similarly,
\begin{gather}\label{form2}
\int_{\partial\Sigma}\psi^2 dL=|\partial\Sigma|_g\sum_{i=1}^k\alpha^2_i+\beta_0^2\int_{\partial\Sigma}u_0^2dL+\int_{\partial\Sigma}\left(\sum_{j=1}^n\beta_ju_j\right)^2dL.
\end{gather}
Plugging~\eqref{form1} and~\eqref{form2} into~\eqref{form}, one gets that $S(\psi,\psi)<0$, as soon as $\Sigma$ is not contained in a
hyperplane in $\mathbb R^{n+1}$ passing through the origin, and since $\sigma_i<\cot r,\,i=1,\ldots,k$. Therefore,
$$
\Ind(\Sigma)\geqslant k+1+n=\Ind_S(\Sigma)+n.
$$
The proof in the hyperbolic case is absolutely similar. In this case we consider the vector space $V=span\{\varphi_1,\ldots,\varphi_k,u_0,u_1,\ldots,u_n\}$ defined in the same way as in the spherical case. Here $u_0$ has eigenvalue $\tanh r$ and $u_i$ have eigenvalues $\coth r$ for all $i=1,\ldots, n$. Further, we take $\psi=\sum_{i=1}^k\alpha_i\varphi_i+\sum_{j=0}^n\beta_ju_j$, for which we get first
\begin{gather}\label{form3}
\int_{\partial\Sigma}\frac{\partial \psi}{\partial \eta}\psi dL=|\partial\Sigma|_g\sum_{i=1}^k\alpha^2_i\sigma_i+\tanh r\,\beta_0^2\int_{\partial\Sigma}u_0^2dL+\coth r\int_{\partial\Sigma}\left(\sum_{j=1}^n\beta_ju_j\right)^2dL
\end{gather}
and second
\begin{gather}\label{form4}
\int_{\partial\Sigma}\psi^2 dL=|\partial\Sigma|_g\sum_{i=1}^k\alpha^2_i+\beta_0^2\int_{\partial\Sigma}u_0^2dL+\int_{\partial\Sigma}\left(\sum_{j=1}^n\beta_ju_j\right)^2dL.
\end{gather}
Plugging~\eqref{form3} and~\eqref{form4} into
$$
S(\psi,\psi)=\int_\Sigma \Big(\Delta_g\psi+(n-1)\psi-|B|^2\psi\Big)\psi dA+\int_{\partial\Sigma}\left(\frac{\partial \psi}{\partial \eta}-\coth r\,\psi\right)\psi dL,
$$
one gets that $S(\psi,\psi)<0$, as soon as $\Sigma$ is not contained in a
hyperplane in $\mathbb R^{n+1}$ passing through the origin, and since $\sigma_i<\coth r,\,i=1,\ldots,k$.
\end{proof}

As a simple corollary of the previous theorem, we obtain

\begin{corollary}\label{cor:n+1}
The index of any free boundary minimal hypersurface in $\mathbb B^n(r)$ in $\mathbb S^n_+$ or $\mathbb H^n$, which is not contained in a
hyperplane in $\mathbb R^{n+1}$ passing through the origin, is at least $n+1$. Moreover, if the index of such a hypersuface if $n+1$, then the spectral index of it is one.
\end{corollary} 

The first part of this statement was proved in~\cite{lima2023eigenvalue} in the spherical case. The second part easily follows from the observation that $\Ind_S(\Sigma)\geqslant 1$.

Yet another corollary is

\begin{corollary}
The critical spherical catenoid in $\mathbb B^3(r)$ in $\mathbb S^n_+$ has index $4$. The index of the critical spherical catenoid in $\mathbb B^3(r)$ in $\mathbb H^3$ is at least 4. 
\end{corollary}

\begin{remark}
We conjecture that the index of the critical spherical catenoid in $\mathbb B^3(r)$ in $\mathbb H^3$ is also 4. 
\end{remark}

This result is a simple combination of Theorems~\ref{thm:ind_M},~\ref{thm:ind+n} and the following one

\begin{theorem}\label{ind_cat}
The spectral index of the critical spherical catenoids in $\mathbb B^3(r)$ in $\mathbb S^3_+$ and $\mathbb H^3$ is one.
\end{theorem}

The proof of this statement repeats the proof for the spherical case given in~\cite{lima2023eigenvalue} up to minor modifications for the hyperbolic case. We postpone it to~Subsection~\ref{sub:ind_cat}.

Finally, we obtain a corollary in the spirit of Corollary 7.3 in~\cite{devyver2019index}.

\begin{corollary}\label{cor:crit}
Let $\Sigma \subset \mathbb B^3(r)$ in $\mathbb S^3_+$ be an FBMS of index $4$, which is not contained in a
hyperplane in $\mathbb R^{4}$ passing through the origin. If $\Sigma$ is a topological annulus then it is the critical spherical catenoid.
\end{corollary}
\begin{proof}
Corollary~\ref{cor:n+1} immediately implies that $\Ind_S(\Sigma)=1$. Then Theorem C in~\cite{lima2023eigenvalue} implies that $\Sigma$ is the critical spherical catenoid.
\end{proof}

\begin{remark}
A similar result for the critical catenoid in a ball in $\mathbb E^3$ was obtained in~\cite[Corollary 7.3]{devyver2019index}. It also follows from the arguments that we provide above, since the critical catenoid is the only FBMS in a ball in  $\mathbb E^3$ with spectral index one (see~\cite[Theorem 1.2]{fraser2016sharp}).
\end{remark}

Corollary~\ref{cor:first} is a combination of Corollaries~\ref{cor:n+1} and \ref{cor:crit}.

\section{Apendix}\label{sec:apendix}

In this section we collect the proofs of the statements that we postponed in the main text. The main purpose of it is to convince the reader that the proofs of analogous statements, which were given in~\cite{lima2023eigenvalue}, also work in the setting of $\mathbb H^n$.

\subsection{Proof of Theorem~\ref{prop:char.min}}\label{sub:char.min}
We explain the proof of part (I) in details and sketch the proof of part (II). 

(I) Consider the subset of $\mathcal{H}_{g}$
$$
\left\{\Big(-\tau\big(\cosh r \, u_0\big)  + \tau\big(\sinh r\, u\big) + \, g,F\big(\cosh r \, u_{0},\sinh r \, u\big)\Big)\right\},
$$
where $u \in V_k(g),\,  ||u||_{L^2(\partial\Sigma,g)}=1$, and take its convex hull $\mathcal{K}$. The Hahn-Banach Theorem implies that $(0,0) \in \mathcal{K}$, i.e., one can find $u_1,\ldots,u_n \in V_{k}(g)$ with $||u_{i}||_{L^2(\partial\Sigma,g)}=1$, and $t_1,\ldots,t_n\in \mathbb R_{+}$ with $\sum_{j=1}^n t_j=1$ such that
\begin{align}\label{eq:extremal2}
\begin{cases}
\displaystyle\sum_{j=1}^n t_j\left(-\tau\big(\cosh r\, u_0\big)  + \tau\big(\sinh r\, u_j\big)+ \,g\right)=0 \  &\textrm{in}\ \Sigma,\\
\displaystyle\sum_{j=1}^n t_j F\left(\cosh r\, u_{0}, \sinh r\, u_{j}\right)=0\ &\textrm{on}\ \partial\Sigma.
\end{cases}
\end{align}
Let $v_0 = \vert\partial\Sigma\vert^{\frac{1}{2}} (\cosh r)u_0$ and $v_j = \left(t_j\vert\partial\Sigma\vert\right)^{\frac{1}{2}} (\sinh r)u_{j},\, j=1,\ldots,n$. Then the first equation of~\eqref{eq:extremal2} implies 

%We can rewrite the first equation as
%$$
%-\tau\big(\cosh r\, u_0\big) + \,g+ \sum_{j=1}^n\tau\big(\sqrt{t_j}(\sinh r)u_j\big) = 0.
%$$

%Denote $v_0 = \vert\partial\Sigma\vert^{\frac{1}{2}} (\cosh r)u_0$ and $v_j = \left(t_j\vert\partial\Sigma\vert\right)^{\frac{1}{2}} (\sinh r)u_{j},\, j=1,\ldots,n$. Hence, by definition of $\tau$, we get
\begin{align}\label{eq:12}
- \frac{1}{2}\Big(\big|\nabla^{g} v_0\big|^{2} + &2v_{0}^{2}\Big)g+dv_0\otimes dv_0+\\ \nonumber &\sum_{j=1}^{n} \left( \frac{1}{2}\Big(\big|\nabla^{g} v_j\big|^{2} +  2v_{j}^{2}\Big)g-dv_j\otimes dv_j\right)+g = 0,
\end{align}
which after taking the trace implies
\begin{equation}\label{eq:sumv}
-v_0^2+\sum_{j=1}^n v_{j}^2 = -1.
\end{equation}
Plugging it into \eqref{eq:12}, we obtain
\begin{equation}
-dv_0\otimes dv_0 +\sum_{j=1}^{n} dv_j\otimes dv_j = \frac{1}{2}\bigg(-\big|\nabla^{g} v_0\big|^{2}+\sum_{j=1}^{n}\big|\nabla^{g} v_j\big|^{2}\bigg)g.
\label{eq:imers}
\end{equation}
Recall that $\Delta_{g}v_j +  2 v_j = 0$. Then one has %~\eqref{eq:sumv} implies 
\begin{align*}
0&=\Delta_g \bigg(-v_{0}^2+\sum_{j=1}^n v_{j}^2\bigg) = -\bigg(2v_0\Delta_g v_0- 2\vert \nabla^g v_0\vert^2\bigg)+\sum_{j=1}^{n}\bigg(2v_j\Delta_g v_j- 2\vert \nabla^g v_j\vert^2\bigg) \\ &=4\bigg(v_{0}^2-\sum_{j=0}^n v_j^2\bigg)+ 2\bigg(\vert \nabla^g v_0\vert^2-\sum_{j=1}^n\vert \nabla^g v_j\vert^2\bigg) \implies -\vert \nabla^g v_0\vert^2+\sum_{j=1}^n\vert \nabla^g v_j\vert^2=2,
\end{align*}
where we also used~\eqref{eq:sumv}. Finally, coming back to~\eqref{eq:imers}, we get
\begin{equation}
-dv_0\otimes dv_0 +\sum_{j=1}^{n} dv_j\otimes dv_j =g.
\label{eq:iso}
\end{equation}
Hence, $v = (v_0,v_1,\ldots,v_n)$ defines an isometric minimal immersion of $\Sigma$ into $\mathbb{H}^{n}$. 

It remains to show that $v(\Sigma) \subset \mathbb B^n(r)$ and $v$ a is free boundary immersion. To this end, we use the second equation of \eqref{eq:extremal2}. It implies
\begin{align*}
0 &= \sum_{j=1}^n t_j F\left(\vert\partial\Sigma\vert^{-\frac{1}{2}}v_0,\left(t_j\vert\partial\Sigma\vert\right)^{-\frac{1}{2}}v_j\right)\\
&=-\frac{\sigma_{0}}{2}(\cosh^{2} r -v_{0}^{2})+\frac{\sigma_k}{2}\left(\sinh^2 r-\sum_{j=1}^n v_{j}^{2}\right)\\
&=-\frac12\sigma_{0}\cosh^{2} r+\frac12\sigma_{k}\sinh^{2}r-\frac12\left(-v_0\frac{\partial v_0}{\partial\eta}+\sum_{j=1}^{n}v_j\frac{\partial v_j}{\partial\eta}\right)\\
&=-\frac12\sigma_{0}\cosh^{2} r+\frac12\sigma_{k}\sinh^{2}r-\frac14\frac{\partial}{\partial\eta}\left(-v_{0}^{2}+\sum_{j=1}^{n}v_{j}^{2}\right)\\
&=-\frac12\sigma_{0}\cosh^{2} r+\frac12\sigma_{k}\sinh^{2}r~\text{ on}~\partial\Sigma,
\end{align*}
where we used in order that $\frac{\partial v_0}{\partial\eta}=\sigma_0 v_0,~\frac{\partial v_j}{\partial\eta}=\sigma_kv_j$ along $\partial\Sigma$ and equation~\eqref{eq:sumv}. Then we see that
\begin{equation}\label{eq:bound.cond}
\sigma_{k}=(\coth^2r)\sigma_0.
\end{equation}
Notice that taking the normal derivative in~\eqref{eq:sumv} yields
\begin{equation}\label{eq:v0}
\sigma_k=(\sigma_k-\sigma_0)v_0^2 \ \mbox{on} \ \partial\Sigma.
\end{equation}
Further, using~\eqref{eq:bound.cond},~\eqref{eq:v0}, and the fact that $v_0$ is positive, we get that $v_{0} = \cosh r $ on $\partial\Sigma$. Moreover, since $v_0$ satisfies $\Delta_{g}v_{0} +  2 v_{0} = 0$, we conclude that the function $v_0$ is subharmonic. Then its maximum is attained on the boundary, i.e., $v_{0} \leqslant \cosh r$ in $\Sigma$. Hence, $v(\Sigma) \subset \mathbb{B}^{n}(r)$.  

Finally, in order to verify that $v$ is free boundary, we apply the tensor fields in \eqref{eq:iso} to $(\eta,\eta)$. We get %we obtain that on $\partial \Sigma$ it holds 
\begin{align*}
1= -\left(\frac{\partial v_0}{\partial \eta}\right)^2+\sum_{j=1}^n\left(\frac{\partial v_j}{\partial \eta}\right)^2= -\sigma_0^2v_0^2&+\sigma_k^2\left(\sum_{j=1}^nv_j^2\right)\\ &=  \sigma_0^2\cosh^2r+\sigma_0^2\coth^4r(\cosh^2r-1),~\text{on}~\partial\Sigma.
\end{align*}
Then $\sigma_0^2=\tanh^2 r$. The variational characterization (see Claim 2) implies that $\sigma_0\geqslant 0$. Then $\sigma_0=\tanh r$ and by~\eqref{eq:bound.cond} $\sigma_k=\coth r$. Hence, $v$ is a free boundary immersion.

(II) Consider of the following subset of $\mathcal H_g$
$$\left\{\Big(-2|\partial\Sigma|_{g}\left(\cosh^2 r\,u_{0}^{2} - \sinh^2 r\,u^{2}-1\right),F\big(\cosh r\,u_{0},\sin r\,u\big)\Big)\right\},$$
where $u \in V_k(g)$. Taking its convex hull $\mathcal K$ and using the Hahn-Banach Theorem we conclude that $(0,0) \in \mathcal{K}$, i.e., there exist $u_1,\ldots,u_n \in V_{k}(g)$, and $t_1,\ldots,t_n\in \mathbb R_+$, such that $\sum_{j=1}^n t_j=1$ and
\begin{align}\label{eq:extremal3}
\begin{cases}
\displaystyle\sum_{j=1}^n t_j\left(\cosh^2 r\,u_{0}^{2} -\sinh^2 r\,u_{j}^{2}-1\right)=0 \  &\textrm{in}\ \Sigma,\\
\displaystyle\sum_{j=1}^n t_j F\left(\cosh r\,u_{0},\sinh r\,u_{j}\right)=0\ &\textrm{on}\ \partial\Sigma.
\end{cases}
\end{align}
Let $v_0 = \cosh r\,u_0$ and $v_j = \sqrt{t_{j}}\sin r\,u_{j},\, j=1,\ldots,n$. Then the first equation of~\eqref{eq:extremal3} yields
$$-v_0^2+\sum_{i=1}^{n}v_{i}^{2} =-1.$$
It shows that $v(\Sigma) \subset \mathbb H^n$. Moreover, $\Delta_{g}v_j + 2v_j =0,\, j=0,1,\ldots,n$ implies that $v\colon (\Sigma, g) \to \mathbb{H}^n$ is a harmonic map. Finally, by the second equation of~\eqref{eq:extremal3} and the same arguments as in the proof of (I), we get that $v(\Sigma) \subset \mathbb B^n(r)$, $v(\partial\Sigma) \subset \partial\mathbb B^n(r)$, and $v$ is a free boundary immersion.

\subsection{Proof of Theorem~\ref{ind_cat}}\label{sub:ind_cat}

The critical spherical catenoid is given by (see~\eqref{H-annulus})
\begin{align*}
&\Phi_0(s,\theta)=\rho(s)\cosh \varphi(s),\\
&\Phi_1(s,\theta)=\rho(s)\sinh\varphi(s),\\
&\Phi_2(s,\theta)=\sqrt{\rho(s)^2-1}\cos\theta,\\
&\Phi_3(s,\theta)=\sqrt{\rho(s)^2-1}\sin\theta,
\end{align*}
where $\rho(s):=\sqrt{a\cosh(2s)+\frac12},\, (s,\theta)\in \Sigma,\, s\in[-s_0,s_0],\,\theta\in[0,2\pi)$. By Proposition~\ref{prop:char}, the coordinate functions of the critical spherical catenoid $\Phi_i, i=0,1,2,3$ are $-2$-Steklov eigenfunctions. Moreover, by Claim 3 $\Phi_0$ is a $\sigma_0=\tanh r$-eigenfunction since $\phi_0$ is positive. In the remaining part of the proof we show that $\Phi_i$, $i=1,2,3$ are $\sigma_1=\coth r$-eigenfunctions. %We will use the same arguments as in the proof of the disk case (see Theorem \ref{thm-disk}).

The Laplacian of the induced metric $g$ takes form
$$-\Delta_gf=\displaystyle \partial^2_{ss}f+\frac{1}{\rho^2-1} \partial^2_{\theta\theta}f+\frac{\rho\rho^{\prime}}{\rho^2-1}\partial_sf.$$
Notice that any $L^2(\Sigma,g)$ can be decomposed as
$$f(s,\theta)=a_0(s)\cdot 1+\displaystyle\sum_{k=1}^{\infty}[a_k(s)\cos(k\theta)+b_k(s)\sin(k\theta)].$$
Then arguing as in the proof of Proposition~\ref{prop:spec_ball}, we conclude that if $f$ is a $\sigma_1$-eigenfunction, then the functions $a_k$ and $b_k$ are identically zero for any $k\geqslant2$ and the functions $a_0,\,a_1\cos\theta,\, b_1\sin\theta$ are $\sigma_1$-eigenfunctions.

For the function $a_0$ one has
$$
 a_0{''}(s)+\frac{\rho(s)\rho^{\prime}(s)}{\rho^2(s)-1}a_0'(s )-2a_0(s)=0.
$$
A direct computation shows that $\Phi_0$ and $\Phi_1$ satisfy this equation. Since the order of the equation is 2, $a_0$ is a linear combination of $\Phi_0$ and $\Phi_1$. Since $\Phi_0$ is a $\sigma_0$-eigenfunction, it implies that if $a_0$ is not identically zero, then $\Phi_1$ is a $\sigma_1$-eigenfunction and $\sigma_1=\coth r$. Hence, $\Phi_2$ and $\Phi_3$ are also $\sigma_1$-eigenfunctions, since they have the same eigenvalue as $\Phi_1$

Consider now the case where $a_0\equiv0$. Then either $a_1$ is not identically zero, or $b_1$ is not identically zero. Without loss of generality, we suppose that $a_1$ is not identically zero. Then $a_1$ satisfies the equation
\begin{equation*}
 a_1{''}(s)+\frac{\rho(s)\rho^{\prime}(s)}{\rho^2(s)-1}a_1'(s )-\left(\frac{1}{\rho^2(s)-1}+2\right)a_1(s)=0.
\end{equation*}
Using the expression for the function $\rho$, one can simplify the previous equation as
\begin{equation}\label{eq:edo-a11}
\displaystyle\left(a\cosh(2s)-\frac{1}{2}\right)a_1''(s)+a\sinh(2s)a_1'(s)-2a\cosh(2s)a_1(s)=0.
\end{equation}
It is not hard to verify that $\sqrt{\rho^2(s)-1}$ satisfies this equation. We look for a second solution to it as $\sqrt{\rho^2(s)-1}h(s)$, where the function $h(s)$ satisfies
$$
3a\sinh(2s)h'(s)+\left(a\cosh(2s)-\frac{1}{2}\right)h''(s)=0.
$$
Solving it, we obtain
$$
h(s)=C_1\int_0^s\left(a\cosh(2t)-\frac{1}{2}\right)^{-\frac{3}{2}}dt +C_2,
$$
for some constants $C_1$ and $C_2$. Then we can take $\{\sqrt{\rho^2(s)-1},\,\sqrt{\rho^2(s)-1}h(s)\}$ as a basis in the space of solutions to the equation~\eqref{eq:edo-a11}. 

Further, we observe that $h$ satisfies
$$
\frac{\dfrac{\partial h}{\partial \eta}(-s_0)}{h(-s_0)}=\frac{\dfrac{\partial h}{\partial \eta}(s_0)}{h(s_0)}=:\mu.
$$
In order to prove it, we need to use that $\frac{\partial}{\partial\eta}=\pm\partial_s$ on $\{\pm s_0\}\times\mathbb S^1$ and $h(-s)=h(s)$. This implies that the functions  $\Phi_2(s,\theta)=\sqrt{\rho^2(s)-1}\cos\theta$ and $\Psi(s,\theta)=\sqrt{\rho^2(s)-1}h(s)\cos\theta$ are eigenfunctions with eigenvalues $\coth r$ and $\coth r+\mu$, respectively. We need to show that $\mu>0$. Then it will imply that $\sigma_1=\coth(r)$. 

One has
$$
\frac{\partial \Phi_2}{\partial s}=\coth r \Phi_2,
$$
which implies that
$$
\coth r=\frac{\rho\rho^{\prime}}{\rho^2-1}.
$$
Similarly,
$$
\frac{\partial \Psi}{\partial s}=(\coth r+\mu)\Psi
$$
yields
\begin{align*}
\coth r+\mu&=\frac{\rho\rho^{\prime}\int_0^{s_0}\left(\rho^2-1\right)^{-\frac{3}{2}}dt+(\rho^2-1)^{-1/2}}{(\rho^2-1)\displaystyle\int_0^{s_0}\left(\rho^2-1\right)^{-\frac{3}{2}}dt}\\
&=\coth r+\frac{1}{(\rho^2-1)^{\frac{3}{2}}\displaystyle\int_0^{s_0}\left(\rho^2-1\right)^{-\frac{3}{2}}dt}.
\end{align*}
Thus, $\mu>0$ and $\Phi_2$ is a $\sigma_1$-eigenfunction with $\sigma_1=\coth r$. Then the functions $\Phi_3(s,\theta)=\sqrt{\rho^2(s)-1}\sin\theta$ and $\Phi_1(s,\theta)=\rho(s)\sinh\varphi(s)$ are also $\sigma_1$-eigenfunctions, since they all have the same eigenvalue. This concludes the proof.

\bibliography{mybib}
\bibliographystyle{alpha}

\end{document}